\documentclass[10pt]{amsart}

\usepackage[english]{babel}

\usepackage[letterpaper,top=2cm,bottom=2cm,left=3cm,right=3cm,marginparwidth=1.75cm]{geometry}

\usepackage{amsmath}
\usepackage{graphicx}
\usepackage{hyperref}
\usepackage{amsmath,arydshln,multirow}
\usepackage[alphabetic]{amsrefs}
\usepackage{arydshln}
\usepackage{cases}
\usepackage{amsmath}
\usepackage{amsfonts}
\usepackage{bm}
\usepackage{arydshln}
\usepackage{amsfonts,amsmath,amssymb,amscd,bbm,amsthm,mathrsfs,dsfont}
\usepackage{mathrsfs}
\usepackage{pb-diagram}
\usepackage{amssymb}
\usepackage[all,cmtip]{xy}
\usepackage{mathtools}
\usepackage{tikz}

\newtheorem{theorem}{Theorem}[section]

\newtheorem{proposition}[theorem]{Proposition}
\newtheorem{lemma}[theorem]{Lemma}

\theoremstyle{definition}
\newtheorem{definition}[theorem]{Definition}
\newtheorem{example}[theorem]{Example}
\newtheorem{proposition-definition}[theorem]{Proposition-Definition}
\newtheorem{definition-theorem}[theorem]{Definition-Theorem}

\newtheorem{corollary}[theorem]{Corollary}

\theoremstyle{remark}
\newtheorem{remark}[theorem]{Remark}

\numberwithin{equation}{section}

\def\mod{\opname{mod}\nolimits}

\def\proj{\opname{proj}\nolimits}
\def\inj{\opname{inj}\nolimits}

\def\Tr{\opname{Tr}\nolimits}

\newcommand{\opname}[1]{\operatorname{\mathsf{#1}}}

\newcommand{\Hom}{\opname{Hom}}
\newcommand{\thick}{\opname{thick}}

\newcommand{\End}{\opname{End}}
\newcommand{\K}{\opname{K}}
\newcommand{\Filt}{\opname{Filt}}

\newcommand{\ftors}{\opname{f.tors}}
\newcommand{\tors}{\opname{tors}}

\newcommand{\Ext}{\opname{Ext}}

\newcommand{\Fac}{\opname{Fac}}
\newcommand{\Sub}{\opname{Sub}}
\newcommand{\add}{\opname{add}\nolimits}

\begin{document}

\title[Relative left Bongartz completions and their compatibility with mutations]{Relative left Bongartz completions and \\ their compatibility with mutations}


\author{Peigen Cao}
\address{Einstein Institute of Mathematics, Edmond J. Safra Campus, The Hebrew University of Jerusalem, Jerusalem  91904, Israel \newline %
Present address: Department of Mathematics   \\
The University of Hong Kong \\
Pokfulam Road               \\
Hong Kong
}
\email{peigencao@126.com}

\author{Yu Wang}
\address{School of Mathematics and Statistics, Taiyuan Normal University, Jinzhong 030619, P. R. China}
\email{dg1621017@smail.nju.edu.cn}

\author{Houjun Zhang}
\address{School of Science, Nanjing University of Posts and Telecommunications, Nanjing 210023, P. R. China}
\email{zhanghoujun@nju.edu.cn}


\dedicatory{Dedicated to Professor Bernhard Keller on the occasion of his 60th birthday}


\date{}



\begin{abstract}
In this paper, we introduce relative left Bongartz completions for a given basic $\tau$-rigid pair $(U,Q)$ in $\mod A$. They give a family of basic $\tau$-tilting pairs containing  $(U,Q)$ as a direct summand. More precisely, to each basic $\tau$-tilting pair $(M,P)$ with $\Fac M\subseteq \prescript{\bot}{}{(\tau U)}\cap Q^\bot$, we
prove that $\Fac U*(\mathcal W\cap\Fac M)$ is a functorially finite torsion class in $\mod A$ such that $\Fac U\subseteq \Fac U*(\mathcal W\cap\Fac M) \subseteq \prescript{\bot}{}{(\tau U)}\cap Q^\bot$, where $\mathcal W:=U^\bot\cap\prescript{\bot}{}({\tau U})\cap Q^\bot$ is the wide subcategory associated to $(U,Q)$. Thus $\Fac U*(\mathcal W\cap\Fac M)$ corresponds to a basic $\tau$-tilting pair $(M^-,P^-)$ containing $(U,Q)$ as a direct summand. We call $(M^-,P^-)$ the left Bongartz completion (or Bongartz co-completion) of $(U,Q)$ with respect to $(M,P)$. Notice that $(M^-,P^-)$ coincides with the usual Bongartz co-completion of $(U,Q)$ in $\tau$-tilting theory if $(M,P)=(0,A)$.
 
 We prove that relative left Bongartz completions have nice 
 compatibility with mutations. Using this compatibility we are able to study the existence of maximal green sequences under $\tau$-tilting reduction. We also explain our construction and some of the results in the setting of silting theory.
\end{abstract}

\maketitle


\tableofcontents

\section{Introduction}
 $\tau$-tilting theory was introduced by Adachi, Iyama and Reiten \cite{adachi_iyama_reiten_2014}  to complete the classical tilting theory from the viewpoint of mutations. Partial tilting modules and tilting modules in classic tilting theory are enlarged to $\tau$-rigid pairs and $\tau$-tilting pairs in $\tau$-tilting theory. Tilting modules may not be able to do mutations in all directions, while $\tau$-tilting pairs can overcome this drawback. That is why it is interesting to study $\tau$-tilting theory.

 In classic tilting theory, Bongartz completion \cite{Bongartz_1980} is an operation to complete a partial tilting module $U$ to the `maximal' tilting module $T$ containing $U$ as a direct summand. Bongartz completion is naturally extended to $\tau$-tilting theory \cite{adachi_iyama_reiten_2014}, which completes a $\tau$-rigid pair $(U,Q)$ to the `maximal' $\tau$-tilting pair containing $(U,Q)$ as a direct summand. Dually, one can also complete $(U,Q)$ to the `minimal' $\tau$-tilting pair  containing $(U,Q)$ as a direct summand, which is called the Bongartz co-completion of $(U,Q)$. Notice that the `minimal' completion, i.e., Bongartz co-completion is not always possible in classic tilting theory. However, this is always possible in $\tau$-tilting theory. 
 
 In this paper, Bongartz completion is renamed by  absolute right Bongartz completion and Bongartz co-completion is renamed by absolute left Bongartz completion. We mainly focus on `left type'=`smaller type' completions. In the final part of this paper, we give some remarks on  `right type'=`bigger type' completions.

Jasso \cite{Jasso_2014} introduced $\tau$-tilting reduction in $\tau$-tilting theory in analogous to Calabi-Yau reduction \cite{Iyama-Yoshino_2008} in cluster tilting theory and silting reduction \cites{Aihara-Iyama_2012,Iyama-Yang_2018} in silting theory. Let $(U,Q)$ be a basic $\tau$-rigid pair in $\mod A$.
The key result of $\tau$-tilting reduction is that we can study the set of basic $\tau$-tilting pairs over $A$ that contain $(U,Q)$ as a direct summand using the set of all basic $\tau$-tilting pairs over a `smaller' algebra $A_{(U,Q)}$. One can see that in $\tau$-tilting reduction, we need to consider the set of basic $\tau$-tilting pairs that contain a particular $\tau$-rigid pair $(U,Q)$ as a direct summand. When the $\tau$-rigid pair $(U,Q)$ is given, the absolute left/right Bongartz completion can only give us the `minimal'/`maximal' $\tau$-tilting pair containing $(U,Q)$ as a direct summand. A natural question is the following:

{\em Question A: How can we get more $\tau$-tilting pairs containing $(U,Q)$ as a direct summand?}

One way is from the viewpoint of mutations. We start with the `minimal'/`maximal' $\tau$-tilting pair containing $(U,Q)$ as a direct summand and use mutations to get more $\tau$-tilting pairs containing $(U,Q)$ as a direct summand. The other way is from the viewpoint of completions. In this case, we need to deform the `absolute left/right Bongartz completion' so that we can get more $\tau$-tilting pairs containing $(U,Q)$ as a direct summand. The first step along this direction is given in \cite{cao-li-2020} in the setting of cluster algebras \cite{fz_2002}.

Cluster algebras were invented by Fomin and Zelevinsky \cite{fz_2002}  as a combinatorial approach to the dual canonical bases of quantized enveloping
algebras. Such algebras have distinguished generators called {\em cluster variables}, which are grouped into overlapping sets of fixed size called {\em clusters}. For the readers who are not familiar with cluster algebras, we propose that one can just view clusters as basic $\tau$-tilting pairs and subsets of clusters as basic $\tau$-rigid pairs.

 In cluster algebras, the first-named author and Li \cite{cao-li-2020} proved that for any cluster ${\bf x}_t$ and any subset $U$ of the `initial' cluster ${\bf x}_{t_0}$, there is a canonical method using $g$-vectors \cite{fomin_zelevinsky_2007} to complete $U$ to a cluster ${\bf x}_{t^\prime}$.
 Notice that there are two inputs here. One is a subset $U$ of the `initial' cluster and the other one is a cluster ${\bf x}_t$. The output is a cluster ${\bf x}_{t^\prime}$ containing $U$ as a subset. We can see when $U$ is fixed, for each cluster ${\bf x}_t$, we can get a completion ${\bf x}_{t^\prime}$ of $U$.

 In \cite{cao_2021}, the first-named author of this paper studied the above completions in cluster algebras from the perspective of cluster tilting theory, silting theory and $\tau$-tilting theory. It turns out that the corresponding completions in silting theory and $\tau$-tilting theory are somehow not natural because of the limitation that $U$ is a subset of the `initial cluster'. This limitation is not a problem in cluster algebras or cluster tilting theory, because we can choose any cluster or cluster tilting object as the initial one. However, this is not the case in silting theory or $\tau$-tilting theory.
 
 Let us explain the reasons in $\tau$-tilting theory. In  $\tau$-tilting theory, the $\tau$-tilting pair $(A,0)$ is always our canonial initial $\tau$-tilting pair. For any basic $\tau$-rigid pair $(U,0)\in\add (A,0)$ and any basic $\tau$-tilting pair $(M,P)$, the author of \cite{cao_2021} completed $(U,0)$ to a basic $\tau$-tilting pair $(M^\prime,P^\prime)$. Namely, when the basic $\tau$-rigid pair $(U,0)\in\add (A,0)$ is fixed,  for each basic  $\tau$-tilting pair $(M,P)$, we get a completion $(M^\prime,P^\prime)$ of $(U,0)$. One can see that the requirement $(U,0)\in\add (A,0)$ is somehow not natural. So  a natural question is the following:
 
 {\em Question B: Fix a general basic $\tau$-rigid pair $(U,Q)$, can we get a completion of $(U,Q)$ for each basic $\tau$-tilting pair $(M,P)$?}
 
 The main motivation of this paper is to study the above question. We found that for any basic $\tau$-tilting pair $(M,P)$ with $\Fac M\subseteq \prescript{\bot}{}({\tau U})\cap Q^\bot$, there does exist a natural completion of $(U,Q)$ with respect to $(M,P)$. This means when the basic $\tau$-rigid pair $(U,Q)$ is fixed, we can get a family of completions of $(U,Q)$. We call them {\em relative left Bongartz completions}.
 
 Now the new viewpoint is there should have two inputs for a completion. One is a basic $\tau$-rigid pair $(U,Q)$ and the other one is a basic $\tau$-tilting pair $(M,P)$ with $\Fac M\subseteq \prescript{\bot}{}({\tau U})\cap Q^\bot$. The role of $(M,P)$ here is to provide some deformation for the absolute left Bongartz completion. There are two very extreme cases such that the condition  $\Fac M\subseteq \prescript{\bot}{}({\tau U})\cap Q^\bot$ is automatically valid. The first one is to make $\Fac M=0$ and the other one is to make $\prescript{\bot}{}({\tau U})\cap Q^\bot=\mod A$. Let us explain them in more details.
 
 Case (i): Take $(M,P)=(0,A)$. Then $\Fac M=0$ and thus $\Fac M\subseteq \prescript{\bot}{}({\tau U})\cap Q^\bot$ is valid for any basic $\tau$-rigid pair $(U,Q)$. In this case, our completion will reduce to the absolute left Bongartz completion (i.e., Bongartz co-completion) in $\tau$-tilting theory. That is why in the absolute left Bongartz completion, we only have one input $(U,Q)$, because the other input is always $(0,A)$ which plays no role in defining the absolute left Bongartz completion.
 
 Case (ii) Take $(U,Q)\in\add (A,0)$. Then $\prescript{\bot}{}({\tau U})\cap Q^\bot=\mod A$ and thus $\Fac M\subseteq \prescript{\bot}{}({\tau U})\cap Q^\bot$ is valid for any $\tau$-tilting pair $(M,P)$. In this case, our completion will reduce to the case in \cite{cao_2021}.
 
 The relative left Bongartz completions in this paper unify the two very extreme cases above and include many other cases. Now we explain why relative left Bongartz completions are interesting. The key reason is that relative left Bongartz completions have nice compatibility with mutations, cf., Theorem \ref{mainthm}. Using this compatibility, we are able to study the existence of maximal green sequences under $\tau$-tilting reduction, cf., Theorem \ref{mainthm2}. The existence of such sequences for finite dimensional Jacobian algebras \cite{DWZ_2008} has  important consequences in cluster algebras, cf., \cite{keller_demonet_2020}.
 
 We also introduce relative left Bongartz completions in the setting of silting theory and prove that  relative left Bongartz completions have nice compatibility with mutations in silting theory, cf., Theorem \ref{mainthm3}. As everyone has already predicted, we prove that the relative left Bongartz completions in silting theory and $\tau$-tilting theory are compatible, cf., Theorem \ref{mainthm4}.
 
 There are two reasons for us to introduce relative left Bongartz completions in silting theory. The first one is to make our theory more complete, because silting theory is closely related with $\tau$-tilting theory. The other one is to make the readers have a better understanding of relative left Bongartz completions from different viewpoints.

 This paper is organized as follows: Section \ref{sec:2} is preliminaries on $\tau$-tilting theory and silting theory. In Section 
 \ref{sec:3}, we define and construct relative left Bongartz completions in $\tau$-tilting theory.  We study the compatibility between relative left Bongartz completions and mutations and use the compatibility to study the existence of maximal green sequences under $\tau$-tilting reduction.   Section \ref{sec:4} serves Section \ref{sec:3} in order to make the readers have an understanding of relative left Bongartz completions from the viewpoint of silting theory. In Section \ref{sec:5}, we give remarks on relative right Bongartz completions in $\tau$-tilting theory and silting theory.

{\em Convention:} In what follows, $K$ always denotes a field. Subcategories in this paper are always full subcategories and closed under isomorphisms. Given two morphisms $f:X\rightarrow Y$ and $g:Y\rightarrow Z$ in some category $\mathcal C$, we denote their composition by $g\circ f=gf$. Given a subcategory $\mathcal X$ of $\mathcal C$, we denote by \begin{eqnarray}
 \mathcal X^\bot&:=&\{C\in \mathcal C| \Hom_{\mathcal C}(\mathcal X, C)=0 \},\nonumber\\
  \prescript{\bot}{}{\mathcal X}&:=&\{C\in \mathcal C  | \Hom_{\mathcal C}(C,\mathcal X)=0 \}.\nonumber
\end{eqnarray}

 {\bf Acknowledgement.} 
 P. Cao is very grateful to Professor David Kazhdan for providing him with a very comfortable working environment and to Yafit-Laetitia Sarfati for her much help in campus life. P. Cao is supported by the European Research Council Grant No. 669655, the National Natural Science Foundation of China Grant No. 12071422, and the Guangdong Basic and Applied Basic Research Foundation Grant No. 2021A1515012035. H. Zhang is supported by Natural Science Research Start-up Foundation
of Recruiting Talents of Nanjing University of Posts and Telecommunications Grant No.
NY222092.

 \section{Preliminaries}\label{sec:2}

\subsection{Approximations in  categories}
Let $\mathcal C$ be a $K$-linear, Krull-Schmidt, additive category and $\mathcal X$ a full subcategory of $\mathcal C$.
Let $C$ be an object in $\mathcal C$. A  {\em right $\mathcal X$-approximation} of $C$ is a  morphism $f:X\rightarrow C$ with $X\in\mathcal X$ such that  for every $X^\prime\in\mathcal X$ 
  the following sequence
  $$\xymatrix{\Hom_{\mathcal C}(X^\prime, X)\ar[rr]^{\Hom_{\mathcal C}(X^\prime,f)}&&\Hom_{\mathcal C}(X^\prime, C)\ar[r]&0}$$
is exact. Dually, a {\em left $\mathcal X$-approximation} of $C$ can be defined.

A subcategory $\mathcal X$ of $\mathcal C$ is said to be {\em contravariantly finite}
in $\mathcal C$ if any object $C$ of $\mathcal C$ has a right $\mathcal X$-approximation.  Dually, we can define {\em covariantly
finite subcategories} in $\mathcal C$. 
A subcategory of $\mathcal C$ is
said to be {\em functorially finite} in $\mathcal C$ if it is both contravariantly and covariantly finite in $\mathcal C$.

We say that a morphism $f:X\rightarrow Y$ is {\em right minimal} if it does not have a direct summand of the form $f_1:X_1 \rightarrow 0$. Dually, we can define {\em left minimal morphisms}.

\subsection{$\tau$-tilting theory}

In this subsection, we recall $\tau$-tilting theory \cite{adachi_iyama_reiten_2014}, which completes the classic tilting theory from the viewpoint of mutations. 

 We fix a finite dimensional basic algebra $A$ over  $K$.  Denote by $\mod A$ the category of finitely generated right $A$-modules, and by $\tau$ the Auslander-Reiten translation in $\mod A$. 

Let $\mathcal C$ be a full subcategory of $\mod A$. There are some full subcategories that are closely related to $\mathcal C$. They are given as follows:
\begin{itemize}
\item $\add\mathcal C$ is the additive closure of $\mathcal C$ in $\mod A$;
\item $\Fac\mathcal C$ consists of the factor modules of the modules in $\add\mathcal C$;
\item $\Sub \mathcal C$ consists of the  submodules of the modules in $\add\mathcal C$;
\item $\Filt \mathcal C$ is the smallest full subcategory of
$\mod A$ containing $\mathcal C$ and closed under extensions.
\end{itemize}

\subsubsection{$\tau$-tilting pair}
Given a module $M\in\mod A$. We denote by $|M|$ the number of non-isomorphic indecomposable direct summands of $M$. We say that $M$ is  \emph{$\tau$-rigid} if $\Hom_A(M,\tau M)=0$.

\begin{definition}[$\tau$-rigid and $\tau$-tilting pair]
(i) Let $M$ be a module in $\mod A$ and $P$ a projective module in $\mod A$. The pair $(M,P)$ is called \emph{$\tau$-rigid} if $M$ is $\tau$-rigid and $\Hom_A(P,M)=0$.

(ii) A $\tau$-rigid pair $(M,P)$ is called \emph{$\tau$-tilting} (resp., \emph{almost $\tau$-tilting}) if $|M|+|P|=|A|$ (resp., $|M|+|P|=|A|-1$).
\end{definition}
Throughout this paper, we always consider $\tau$-rigid pairs up to isomorphisms.

A $\tau$-rigid pair $(M,P)$ is \emph{indecomposable} if $M\oplus P$ is indecomposable in $\mod A$ and it is \emph{basic} if both $M$ and $P$ are basic in $\mod A$. 
A module $M\in \mod A$ is called a \emph{support $\tau$-tilting module} if there exists a basic projective module $P$ such that $(M,P)$ is a $\tau$-tilting pair. In fact, such a basic projective module $P$ is unique up to isomorphisms \cite{adachi_iyama_reiten_2014}*{Proposition 2.3}.

Let $(M,P)$ and $(U,Q)$ be two $\tau$-rigid pairs in $\mod A$. We say that $(U,Q)$ is a {\em direct summand} of $(M,P)$, if there exists another  $\tau$-rigid pair $(U',Q')$ such that $M\cong U\oplus U'$ and $P\cong Q\oplus Q'$.

\begin{theorem}[\cite{adachi_iyama_reiten_2014}*{Theorem 2.12 and Theorem 2.18}] \label{thmair}
Let $(U,Q)$ be a basic almost $\tau$-tilting pair in $\mod A$. Then $\Fac U\subsetneq \prescript{\bot}{}{(\tau U)}\cap Q^\bot$ and
there exist exactly two basic
$\tau$-tilting pairs $(M,P)$ and $(M^\prime,P^\prime)$ containing $(U,Q)$ as a direct summand. Moreover,
$$\{\Fac M,\Fac M^\prime\}=\{\Fac U,\prescript{\bot}{}{(\tau U)}\cap Q^\bot\}.$$
In particular, either $\Fac M\subsetneq \Fac M^\prime$ or
$\Fac M^\prime\subsetneq\Fac M$ holds.
\end{theorem}

\begin{definition}[Left and right mutation]
Keep the notations in Theorem \ref{thmair}. The operation $(M,P)\mapsto(M^\prime,P^\prime)$ is called a \emph{mutation} of $(M,P)$. If $\Fac M\subsetneq \Fac M^\prime$ holds, we call $(M^\prime, P^\prime)$ a {\em right mutation} of $(M,P)$. If $\Fac M^\prime \subsetneq \Fac M$ holds, we call $(M^\prime, P^\prime)$ a {\em left mutation} of $(M,P)$.
\end{definition}

\begin{remark}
We remark that one can look at the sequence of integers $\ldots,-2,-1,0,1,2,\ldots$. Going left means becoming smaller and going right means becoming bigger. Then one can easily remember the difference between left and right mutation. This also applies to the left and right Bongartz completion defined later. Another reason using the word `left' is that `left type' completions are constructed using left approximations in the setting of silting theory, cf., Section \ref{sec:4}. 
\end{remark}

\subsubsection{Functorially finite torsion classes}

Let $\mathcal A$ be an abelian category. For example, $\mathcal A=\mod A$ or we take $\mathcal A$ some {\em wide subcategory} of $\mod A$, which is a full subcategory
of $\mod A$ closed under kernels, cokernels and extensions. A \emph{torsion pair} $(\mathcal{T},\mathcal{F})$ in  $\mathcal A$ is a pair of subcategories of $\mathcal A$ satisfying that
\begin{itemize}
\item[(i)] $\Hom_{\mathcal A}(T,F)=0$ for any $T\in \mathcal{T}$ and $F\in \mathcal{F}$;
\item[(ii)] for any object $X\in \mathcal A$, there exists a short exact sequence
 $$0\rightarrow X_t\rightarrow X\rightarrow X_f\rightarrow0$$ with $X_t\in\mathcal{T}$ and $X_f\in \mathcal{F}$. Thanks to the condition (i), such a sequence is unique up to isomorphisms. This short exact sequence is called the \emph{canonical sequence} of $X$ with respect to $(\mathcal{T},\mathcal{F})$.
\end{itemize}

Notice that in a torsion pair  $(\mathcal T,\mathcal F)$, we always have $\mathcal F=\mathcal T^\bot$ and $\mathcal T= \prescript{\bot}{}{\mathcal F}$.

The subcategory $\mathcal T$ (resp., $\mathcal F$) in a torsion pair $(\mathcal T,\mathcal F)$ is called a \emph{torsion class} (resp., \emph{torsion-free class}) in $\mathcal A$. Let $\mathcal C$ be a full subcategory of $\mathcal A$. It is well known that $\mathcal C$ is a torsion class if and only if it is closed under extensions and epimorphisms in $\mathcal A$ and $\mathcal C$ is a torsion-free class if and only if it is closed under extensions and monomorphisms in $\mathcal A$.

Using the definition of torsion pair, one can easily show that a torsion class is always contravariantly finite in  $\mathcal A$ and a torsion-free class is always covariantly finite in $\mathcal A$.
A torsion class $\mathcal T$ (respectively, a torsion-free class $\mathcal F$) in $\mathcal A$ is said to be \emph{functorially finite} if it is a functorially finite subcategory of $\mathcal A$.

Now we focus on the case $\mathcal A=\mod A$. Let $\mathcal C$ be a subcategory of $\mod A$. A module $M\in\mathcal C$ is said to be {\em Ext-projective} in $\mathcal C$, if $\Ext_A^1(M,\mathcal C)=0$.  Auslander-Smalø \cite{AS_1981} give a sufficient condition on $\mathcal C$ such that it has only finitely many indecomposable Ext-projective modules up to isomorphisms.

\begin{proposition} [\cite{AS_1981}*{Corollary 4.4}] Suppose that $\mathcal C$ is a subcategory of $\mod A$ closed under extensions and there exists $M\in\mathcal C$ such that $\mathcal C=\Fac M$. Then there are only finitely many indecomposable Ext-projective modules in $\mathcal C$ up to isomorphisms.  
\end{proposition}

Keep the assumptions in the above proposition. We denote by $\mathcal P(\mathcal C)$ the direct sum of one copy of each of the indecomposable Ext-projective objects in $\mathcal C$ up to isomorphisms. Notice that $\mathcal P(\mathcal C)$ still belongs to $\mathcal C$, thanks to the above proposition.

\begin{theorem}[\cite{adachi_iyama_reiten_2014}*{{Proposition 1.2 (b) and Theorem 2.7}}]\label{proorder} The following statements hold.
\begin{itemize}
    \item [(i)] There is a well-defined map $\Psi$ 
from $\tau$-rigid pairs to functorially finite torsion classes in $\mod A$
given by $(M,P)\mapsto \Fac M$.

\item[(ii)] The above map $\Psi$ is a bijection if we restrict it to basic $\tau$-tilting pairs, which induces a natural order on the set of basic $\tau$-tilting pairs 
$$(M,P)\leq(M^\prime,P^\prime)\overset{{\rm def.}}{\Longleftrightarrow}\Fac M\subseteq\Fac M^\prime.$$

\item[(iii)] Let $\mathcal T$ be a functorially finite torsion class and denote by $(M,P)=\Psi^{-1}(\mathcal T)$ the basic $\tau$-tilting pair corresponding to $\mathcal T$. Then  $M=\mathcal P(\mathcal T)$.
\end{itemize}
\end{theorem}

\begin{proposition}[\cite{adachi_iyama_reiten_2014}*{Proposition 2.9 and Theorem 2.10}]\label{prominmax}
 Let $(U,Q)$ be a basic $\tau$-rigid pair and  $\mathcal T$ a functorially finite torsion class in $\mod A$. Then
 \begin{itemize}
 \item[(i)]  $\Fac U$ and  $\prescript{\bot}{}({\tau U})\cap Q^\bot$ are functorially finite torsion classes in $\mod A$;
 \item[(ii)]  $\Fac U\subseteq \mathcal T\subseteq \prescript{\bot}{}({\tau U})\cap Q^\bot$ if and only if $(U,Q)$ is a direct summand of the basic $\tau$-tilting pair $\Psi^{-1}(\mathcal T)$, where $\Psi$ is the bijection given in Theorem \ref{proorder}.
 \end{itemize}
\end{proposition}

\begin{definition}[Absolute left/right Bongartz completion]\label{def:comp}
Let $(U,Q)$ be a basic $\tau$-rigid pair in $\mod A$ and  $\Psi$ the bijection in Theorem \ref{proorder}.
\begin{itemize}
\item[(i)] The basic $\tau$-tilting pair $\Psi^{-1}(\Fac U)$ is
called the {\em absolute left Bongartz completion} (or {\em Bongartz co-completion}) of $(U,Q)$.
 \item[(ii)] The basic $\tau$-tilting pair $\Psi^{-1}(\prescript{\bot}{}({\tau U})\cap Q^\bot)$ is called the {\em absolute right Bongartz completion}  (or {\em Bongartz completion}) of $(U,Q)$.
\end{itemize}
\end{definition}

\subsubsection{$\tau$-tilting reduction}
We recall the $\tau$-tilting reduction introduced by Jasso \cite{Jasso_2014}. Let $(U,Q)$ be a $\tau$-rigid pair in $\mod A$. The key result of $\tau$-tilting reduction is that we can study the set of basic $\tau$-tilting pairs over $A$ that contain $(U,Q)$ as a direct summand using the set of all basic $\tau$-tilting pairs over a `smaller' algebra $A_{(U,Q)}$.

Now we give a brief summary of the construction of $A_{(U,Q)}$. Let $(M,P)$ be the absolute right Bongartz completion of $(U,Q)$, namely, $\Fac M=\prescript{\bot}{}({\tau U})\cap Q^\bot$. Notice that $(M,P)$ actually has the form: $P=Q$ and $M=U\oplus U^\prime$ for some $\tau$-rigid module $U^\prime$ in $\mod A$.

Denote by $B=\End_A(M)=\End_A(U\oplus U^\prime)$ the endomorphism algebra of $M=U\oplus U^\prime$. In the algebra $B=\End_A(M)$, there exists an idempotent element $e_U$ corresponding to the projective right $B$-module $\Hom_A(M,U)$. We define the algebra $A_{(U,Q)}$ as follows: 
$$A_{(U,Q)}:=B/Be_UB.$$
This algebra is called the {\em $\tau$-tilting reduction} of $A$ at the $\tau$-rigid pair $(U,Q)$.

\begin{theorem}[\cite{Jasso_2014}*{Theorem 1.1}]\label{thm:taureduction}
Let  $(U,Q)$ be a basic $\tau$-rigid pair in $\mod A$ and $A_{(U,Q)}$ the $\tau$-tilting reduction of $A$ at $(U,Q)$. Then
there is an order-preserving
bijection between the set of basic $\tau$-tilting pairs in $\mod A$ that
 contain $(U,Q)$ as a direct summand and the set of all basic
$\tau$-tilting pairs in  $\mod A_{(U,Q)}$. 
\end{theorem}

\subsection{Silting theory}
In this subsection, we fix a $K$-linear, Krull-Schmidt, Hom-finite triangulated category $\mathcal D$ with a silting object. 

Recall that an object $U\in\mathcal D$ is {\em presilting} if $\Hom_{\mathcal D}(U,U[i])=0$ for any  integer $i>0$. A presilting
object $S\in\mathcal D$ is {\em silting} if 
$\mathcal D$ coincides with the thick subcategory of $\mathcal D$ generated by $S$, i.e., $\mathcal D=\thick(S)$. A presilting object $U\in\mathcal D$ is {\em almost silting} if there
exists an indecomposable object $X\notin\add U$ such that $U\oplus X$ is silting. In what follows, we always consider presilting objects up to isomorphisms.

\begin{theorem}[\cite{Aihara-Iyama_2012}*{Theorem 2.27}]
\label{thm:k0}
Let $S=\bigoplus_{i=1}^nS_i$ be a basic silting object in $\mathcal D$, where each $S_i$ is indecomposable. Then the Grothendieck group $\K_0(\mathcal D)$ of $\mathcal D$ is a free
abelian group with a basis $[S_1],\ldots,[S_n]$. In particular, any two basic silting objects in $\mathcal D$ have the same number of indecomposable direct summands.
\end{theorem}

\begin{definition-theorem}[Left and right mutation, \cite{Aihara-Iyama_2012}]
Let $S=\bigoplus_{i=1}^nS_i$ be a basic silting object in $\mathcal D$. Fix a $k\in\{1,\ldots,n\}$ and write $S=S_k\oplus U$, where $U=\bigoplus_{i\neq k}S_i$. Take a triangle
$$\xymatrix{S_k\ar[r]^f&U^\prime \ar[r]&X\ar[r]&S_k[1]}$$
with a minimal left $\add U$-approximation $f:S_k\rightarrow U^\prime$ of $S_k$. Then $U\oplus X$ is silting in $\mathcal D$. We call the basic silting object $(U\oplus X)^\flat$ corresponding to $U\oplus X$ the {\em left mutation} of $S$ at $S_k$ and denote it by $\mu_{S_k}^L(S):=(U\oplus X)^\flat$.
Dually, we can define the {\em right mutation}  $\mu_{S_k}^R(S)$ of $S$ at $S_k$.
\end{definition-theorem}

Notice that the left mutation defined above corresponds to the irreducible left mutation in \cite{Aihara-Iyama_2012}. In this paper, we only consider irreducible left/right mutations.

Let $X$ and $Y$ be any two objects in $\mathcal D$. We write $Y\leq X$ if $\Hom_{\mathcal D}(X,Y[i])=0$ for any  integer $i>0$. We write $Y<X$ if $Y\leq X$ and $Y\neq X$. Notice that we  have $U[1]\leq U$ for any presilting object $U$.

\begin{theorem}[\cite{Aihara-Iyama_2012}*{Theorem 2.11 and Theorem 2.35}]
\label{thm:AIcover}
The following statements hold.
\begin{itemize}
    \item [(i)] The relation ``\;$\leq$\;" induces a partial order on the set of basic silting objects in $\mathcal D$.
    \item[(ii)] Let $T$ and $S$ be two basic silting objects in $\mathcal D$. Then $T$ is a left mutation of $S$ if and only if $T<S$ and there is no basic silting object $M$ such that $T<M<S$.
\end{itemize}
\end{theorem}

Let $S$ be a basic silting object in $\mathcal D$. Denote by 
$S\ast S[1]$ the full subcategory of $\mathcal D$ consisting of all the object $Z\in\mathcal D$ such that there exists a triangle
$$X\rightarrow Z\rightarrow Y\rightarrow X[1]$$
with $X\in \add S$ and $Y\in \add (S[1])$.

\begin{theorem}\label{thm:OIY14}
Let $S$ be a basic silting object in $\mathcal D$ and denote by $\mathcal C=S\ast S[1]$. Let
$A=\End_{\mathcal D}(S)$ be the endomorphism algebra of $S$ in $\mathcal D$. Then the following statements hold.
\begin{itemize}
    \item [(i)] {\rm (}\cite{Jasso_2014}*{Proposition 4.5}, \cite{Iyama-Yoshino_2008}{\rm).} The functor $F=\Hom_{\mathcal D}(S,-):\mathcal C\rightarrow \mod A$ induces an equivalence of abelian categories$$\overline{(-)}:\frac{\mathcal C}{[S[1]]}\rightarrow\mod A,$$
    where $[S[1]]$ is the ideal of $\mathcal C$ consisting of morphisms that factor through $\add (S[1])$.
    \item[(ii)] {\rm (}\cite{Jasso_2014}, \cite{Iyama-Jorgensen-Yang_2014}*{Theorem 4.5}{\rm).} The functor $F$ induces a bijection from presilting objects in $\mathcal C$ to $\tau$-rigid pairs in $\mod A$ via
    $$U=X_U\oplus Y_U[1]\mapsto (F(X_U),F(Y_U))=(\overline X_U,\overline Y_U),$$
    where $U$ is a  presilting object in $\mathcal C$ and $U=X_U\oplus Y_U[1]$ is a direct sum decomposition of $U$ such that $Y_U[1]$ is a maximal direct summand of $U$ that
belongs to $\add (S[1])$. Moreover, this bijection restricts to an order-preserving bijection from the poset of basic silting objects in $\mathcal C$ to the poset of basic $\tau$-tilting pairs in $\mod A$. 
\end{itemize}

\end{theorem}

\section{Relative left Bongartz completions in $\tau$-tilting theory}\label{sec:3}

\subsection{Relative left Bongartz completions}

For any two full subcategories $\mathcal X$ and $\mathcal Y$ of $\mod A$, we denote by $\mathcal X\ast \mathcal Y$ the full subcategory of $\mod A$ consisting of the modules $M$ such that there exists a short exact sequence 
$$\xymatrix{0\ar[r]&X\ar[r]&M\ar[r]&Y\ar[r]&0}$$
with $X\in\mathcal X$ and $Y\in\mathcal Y$.

Let $\mathcal A$ be an abelian category. We denote by 
\begin{itemize}
    \item  $\tors(\mathcal A)$ the set of  torsion classes in $\mathcal A$;

    \item $\ftors(\mathcal A)$ the set of  functorially finite torsion classes in $\mathcal A$. 

\end{itemize}
 If $\mathcal A=\mod A$, we simply write  $\tors(A):=\tors(\mod A)$ and
 $\ftors(A):=\ftors(\mod A)$.

For a basic $\tau$-rigid pair $(U,Q)$ in $\mod A$, we denote by
\begin{eqnarray}
\tors_{(U,Q)}(A)&:=&\{\mathcal T\in \tors(A)|\Fac U\subseteq \mathcal T \subseteq \prescript{\bot}{}({\tau U})\cap Q^\bot\},\nonumber\\
\ftors_{(U,Q)}(A)&:=&\{\mathcal T\in \ftors(A)|\Fac U\subseteq \mathcal T \subseteq \prescript{\bot}{}({\tau U})\cap Q^\bot\}.\nonumber
\end{eqnarray}

\begin{theorem}\label{thm:Jasso}
Let $(U,Q)$ be a basic $\tau$-rigid pair in $\mod A$ and $(M,P)$ the absolute right Bongartz completion of $(U,Q)$. Denote by $B=\End_A(M)$ and $A_{(U,Q)}=B/Be_UB$ the $\tau$-tilting reduction of $A$ at $(U,Q)$. Then the following statements hold.

\begin{itemize}
    \item [(i)] {\rm (}\cite{Yurikusa_2018},\cite{bst_2019}*{Corollary 3.25}{\rm).} The subcategory $\mathcal W:=U^\bot\cap\prescript{\bot}{}({\tau U})\cap Q^\bot$ is a wide subcategory of $\mod A$, that is, $\mathcal W$ is closed under kernels, cokernels and extensions.
    
    \item[(ii)] {\rm (}\cite{Jasso_2014}*{Theorem 3.8}{\rm).} The functor $F=\Hom_A(M,-): \mod A\rightarrow \mod B$ restricts to an equivalence of abelian categories from $\mathcal W$ to $\mod A_{(U,Q)}\subseteq\mod B$.
    
    \item[(iii)] {\rm (}\cite{Jasso_2014}*{Corollary 3.11}{\rm).}
    The functor $F$ induces an order-preserving bijection  from
     $\tors(\mathcal W)$ to  $\tors(A_{(U,Q)})$, which restricts to an order-preserving bijection  from
 $\ftors(\mathcal W)$ to  $\ftors(A_{(U,Q)})$.

    \item[(iv)] {\rm (}\cite{Jasso_2014}*{Theorem 3.12 and Theorem 3.14}{\rm).} The map $\tors(\mathcal W)\ni\mathcal G\mapsto \Fac U
\ast\mathcal G $
is an order-preserving bijection from $\tors(\mathcal W)$ to $\tors_{(U,Q)}(A)$ and it induces an order-preserving bijection from $\ftors(\mathcal W)$ to $\ftors_{(U,Q)}(A)$.
\end{itemize}
\end{theorem}

 \begin{remark}
 For $\mathcal G\in\tors(\mathcal W)$, the subcategory $\Fac U
\ast\mathcal G$ is actually the smallest torsion class in $\mod A$  containing both $U$ and $\mathcal G$, cf. \cite{Jasso_2014}*{Remark 3.13}.
 \end{remark}

\begin{definition} [Relative left Bongartz completion] \label{def:left}
Let $(U,Q)$ be a basic $\tau$-rigid pair and $\mathcal W:=U^\bot\cap\prescript{\bot}{}({\tau U})\cap Q^\bot$ the associated wide subcategory of $\mod A$. Let $(M,P)$ be a basic $\tau$-tilting pair in $\mod A$ such that $\Fac U\ast(\mathcal W\cap\Fac M)$ is a functorially finite torsion class in $\ftors_{(U,Q)}(A)$.
 Then by Proposition \ref{prominmax}, we know that $(U,Q)$ is a direct summand of the basic $\tau$-tilting pair $(M^-,P^-)$ corresponding to the functorially finite torsion class $\Fac U\ast(\mathcal W\cap\Fac M)$ in $\mod A$. We call $(M^-,P^-)$ the {\em left Bongartz completion} of $(U,Q)$ with respect to $(M,P)$ and denote it by  $$B^-_{(U,Q)}(M,P):=(M^-,P^-).$$
\end{definition}

In the following proposition, we will show that $\Fac U\ast(\mathcal W\cap\Fac M)$  is always a torsion class in $\tors_{(U,Q)}(A)$
and we give a sufficient condition  such that it is functorially finite.

\begin{proposition} \label{pro:finite} Let $(U,Q)$ be a basic $\tau$-rigid pair and $\mathcal W:=U^\bot\cap\prescript{\bot}{}({\tau U})\cap Q^\bot$ the associated wide subcategory of $\mod A$. Let  $(M,P)$ be a basic $\tau$-tilting pair in $\mod A$. Then the following statements hold.
\begin{itemize}
\item[(i)] $\mathcal W\cap \Fac M$ is a torsion class in $\tors(\mathcal W)$
 and $\Fac U\ast(\mathcal W\cap\Fac M)$  is a torsion class in  $\tors_{(U,Q)}(A)$.

\item[(ii)] Any $W\in\mathcal W$ has a left $U^\bot\cap \Fac M$-approximation.
    
    \item[(iii)]  The following conditions are equivalent.
    \begin{itemize}
        \item [(a)]  $U^\bot\cap\Fac M$ equals $\mathcal W\cap \Fac M$;
        \item[(b)] $U^\bot\cap\Fac M$ is a subcategory of $\mathcal W$;
        \item[(c)] $U^\bot\cap\Fac M$ is a subcategory of $\prescript{\bot}{}({\tau U})\cap Q^\bot$;
        \item[(d)] $\Fac M$ is a subcategory of $\prescript{\bot}{}({\tau U})\cap Q^\bot$.
    \end{itemize}
    \end{itemize}
Now we suppose that $\Fac M$ is a subcategory of $\prescript{\bot}{}({\tau U})\cap Q^\bot$. Then we further have
    \begin{itemize}
    \item[(iv)]   $\mathcal W\cap \Fac M$ is a functorially finite torsion class in $\ftors(\mathcal W)$ and $\Fac U\ast(\mathcal W\cap\Fac M)$ is a functorially finite torsion class in  $\ftors_{(U,Q)}(A)$. 
         \item[(v)] The left Bongartz completion $B_{(U,Q)}^-(M,P)$ of $(U,Q)$ with respect to $(M,P)$ exists and we have
         $$\Fac U\ast(\mathcal W\cap \Fac M)=\Fac U\ast(U^\bot\cap\Fac M)=\Fac U\ast\Fac M.$$
\end{itemize}
\end{proposition}
\begin{proof}

(i) It is straightforward to check that $\mathcal W\cap\Fac M$ is closed under extensions and epimorphisms in $\mathcal W$. So $\mathcal W\cap\Fac M$ is a torsion class in  $\tors(\mathcal W)$. Then by Theorem \ref{thm:Jasso} (iv), we know that $\Fac U\ast(\mathcal W\cap\Fac M)$  is a torsion class in $\tors_{(U,Q)}(A)$.

(ii)  Fix a module  $W\in\mathcal W$. Since $\Fac M$ is a functorially finite torsion class in $\mod A$, there exists a morphism  $f:W\rightarrow N$ in $\mod A$ such that $f$ is a left $\Fac M$-approximation of $W$. Let 
 $$\xymatrix{0\ar[r]&N_t\ar[r]^i&N\ar[r]^p&N_f\ar[r]&0}$$
 be the canonical sequence of $N$ with respect to the torsion pair $(\Fac U, U^\bot)$ in $\mod A$, where $N_t\in \Fac U$ and $N_f\in U^\bot$.
 Since $N$ is in $\Fac M$ and $\Fac M$ is closed under quotient modules, we get $N_f\in \Fac M$ and thus $N_f\in U^\bot\cap\Fac M$. 
 
 Now we show that $pf:W\rightarrow N_f$ is a left $U^\bot\cap\Fac M$-approximation of $W$ by explaining the following diagram
 $$\xymatrix{0\ar[r]&N_t\ar[r]^i&N\ar[r]^p\ar[rd]^{g_1}&N_f\ar[r]\ar[d]^{g_2}&0.\\
 &&W\ar[u]^f\ar[r]^g&X&}$$
 For any $X\in U^\bot\cap\Fac M$ and any morphism $g: W\rightarrow X$, we need to show that $g$ factors through $pf$. Since $X$ is in $\Fac M$ and $f:W\rightarrow N$ is a left $\Fac M$-approximation of $W$, we know that $g:W\rightarrow X$ factors through $f$, i.e., there exists a morphism $g_1:N\rightarrow X$ such that $g=g_1 f$. Since $N_t\in \Fac U$  and $X\in U^\bot$, we have $g_1 i=0$. Thus $g_1$ factors through $p$, which is the cokernel of $i$. Hence, there exists $g_2: N_f\rightarrow X$ such that $g_1=g_2 p$. So we have $$g=g_1 f=g_2 pf.$$
 So $g$ factors through $pf$ and thus $pf:W\rightarrow N_f$ is a left $U^\bot\cap\Fac M$-approximation of $W$.

(iii) The equivalences ``$(a)\Longleftrightarrow(b)\Longleftrightarrow(c)$" are clear, because $\mathcal W=U^\bot\cap\prescript{\bot}{}({\tau U})\cap Q^\bot$.

``$(c)\Longrightarrow(d)$": Assume that $U^\bot\cap\Fac M$ is a subcategory of $\prescript{\bot}{}({\tau U})\cap Q^\bot$. We need to prove $\Fac M$ is a subcategory of $\prescript{\bot}{}({\tau U})\cap Q^\bot$. For any $L\in\Fac M$, let 
$$0\rightarrow L_t\rightarrow L \rightarrow L_f \rightarrow 0$$
be the canonical sequence of $L$ with respect to the torsion pair $(\Fac U, U^\bot)$, where $L_t\in\Fac U$ and $L_f\in U^\bot$. Since $\Fac M$ is closed under quotient modules and by $L\in \Fac M$, we get $L_f\in\Fac M$ and thus $L_f\in U^\bot\cap \Fac M$.
Now we have $L_t\in\Fac U\subseteq \prescript{\bot}{}({\tau U})\cap Q^\bot$ and $L_f\in  U^\bot\cap \Fac M\subseteq \prescript{\bot}{}({\tau U})\cap Q^\bot$.
Since $\prescript{\bot}{}({\tau U})\cap Q^\bot$ is a torsion class in $\mod A$, which is closed under extensions, we get $L\in \prescript{\bot}{}({\tau U})\cap Q^\bot$. Hence, $\Fac M$ is a subcategory of $\prescript{\bot}{}({\tau U})\cap Q^\bot$.

``$(d)\Longrightarrow(c)$": This is  clear. Hence, the conditions in (iii) are equivalent.

(iv) We first show that $\mathcal W\cap\Fac M\in\ftors(\mathcal W)$. By (i), we have $\mathcal W\cap\Fac M\in \tors(\mathcal W)$. It still needs to check that $\mathcal W\cap\Fac M$ is  functorially finite in $\mathcal W$. Since $\mathcal W\cap\Fac M$ is a torsion class in $\mathcal W$, it suffices to check that $\mathcal W\cap\Fac M$ is covariantly finite in $\mathcal W$, that is, we need to show that any $W\in\mathcal W$ has a left $\mathcal W\cap\Fac M$-approximation.

Since $\Fac M$ is a subcategory of $\prescript{\bot}{}({\tau U})\cap Q^\bot$ and by (iii), we have the equality  $$U^\bot\cap\Fac M=\mathcal W\cap \Fac M.$$
By (ii), we know that $U^\bot\cap\Fac M=\mathcal W\cap \Fac M$ is a  covariantly finite subcategory of $\mathcal W$. Hence, $\mathcal W\cap \Fac M$ is a functorially finite torsion class in $\ftors(\mathcal W)$. 

By Theorem \ref{thm:Jasso} (iv), we get that $\Fac U\ast(\mathcal W\cap\Fac M)$  is a functorially finite torsion class in $\ftors_{(U,Q)}(A)$.

 (v) By (iv) and
 the definition of relative left Bongartz completion, we know that the left Bongartz completion $B_{(U,Q)}^-(M,P)$ of $(U,Q)$ with respect to $(M,P)$ exists and it corresponds to the functorially finite torsion class $\Fac U\ast(\mathcal W\cap \Fac M)$ in $\mod A$.

Since  $\Fac M$ is a subcategory of $\prescript{\bot}{}({\tau U})\cap Q^\bot$ and by (iii), we have  $U^\bot\cap\Fac M=\mathcal W\cap \Fac M$. Thus the equality $\Fac U\ast(\mathcal W\cap \Fac M)=\Fac U\ast(U^\bot\cap\Fac M)$ holds.
The proof of the equality $\Fac U\ast\Fac M=\Fac U\ast(U^\bot\cap\Fac M)$ is similar to the proof of \cite{cao_2021}*{Proposition 6.8}.
\end{proof}

\begin{remark} By Proposition \ref{pro:finite} (iv)(v), the inclusion $\Fac M\subseteq\prescript{\bot}{}({\tau U})\cap Q^\bot$ is a sufficient condition such that $\Fac U\ast(\mathcal W\cap\Fac M)$ is a functorially finite torsion class in $\ftors_{(U,Q)}(A)$ and we have  $\Fac U\ast(\mathcal W\cap\Fac M)=\Fac U\ast \Fac M$ in this case.
There are two very extreme cases such that the inclusion $\Fac M\subseteq\prescript{\bot}{}({\tau U})\cap Q^\bot$ is automatically valid.
\begin{itemize}

    \item [(i)] If we take $(M,P)=(0,A)$, then $\Fac M=0$ and thus
    the inclusion $0=\Fac M\subseteq\prescript{\bot}{}({\tau U})\cap Q^\bot$ holds for any basic $\tau$-rigid pair $(U,Q)$.
      In this case, the left Bongartz completion $B_{(U,Q)}^-(0,A)$ of $(U,Q)$ with respect to $(0,A)$ can be defined and it is just the basic $\tau$-tilting pair corresponding to the functorially finite torsion class $\Fac U$ in $\mod A$. Hence, it coincides with the absolute left Bongartz completion of $(U,Q)$ given in Definition \ref{def:comp}.
    \item[(ii)]  If we take $(U,Q)\in \add (A,0)$, then $\prescript{\bot}{}({\tau U})\cap Q^\bot=\mod A$ and thus
      the inclusion $\Fac M\subseteq\prescript{\bot}{}({\tau U})\cap Q^\bot$ holds for any basic $\tau$-tilting pair $(M,P)$.   In this case, the left Bongartz completion $B_{(U,Q)}^-(M,P)$ of $(U,Q)$ with respect to $(M,P)$ can be defined and it coincides with the one introduced in \cite{cao_2021}*{Definition 6.9} by the first named author of this paper. 
\end{itemize}

The relative left Bongartz completion in Definition \ref{def:left} unify the two very extreme cases in (i) and (ii) and include many other cases.
\end{remark}

\begin{corollary}\label{cortrivial}
Let $(M,P)$ be a basic $\tau$-tilting pair in $\mod A$ and $(U,Q)$ a direct summand of $(M,P)$. Then the left Bongartz completion $B_{(U,Q)}^-(M,P)$ of $(U,Q)$ with respect to $(M,P)$ exists and we have $B_{(U,Q)}^-(M,P)=(M,P)$.
\end{corollary}
\begin{proof}
Since $(U,Q)$ is a direct summand of $(M,P)$ and by Proposition \ref{prominmax}, we have
$$\Fac U\subseteq \Fac M \subseteq \prescript{\bot}{}({\tau U})\cap Q^\bot.$$
 Then by Proposition \ref{pro:finite} (v), the left Bongartz completion $B_{(U,Q)}^-(M,P)$ of $(U,Q)$ with respect to $(M,P)$ exists and $B_{(U,Q)}^-(M,P)$ is the basic $\tau$-tilting pair corresponding to the functorially finite torsion class $\Fac U\ast\Fac M$. By $\Fac M\subseteq \Fac U\ast\Fac M\subseteq \Fac M\ast \Fac M=\Fac M$, 
we know that the two basic $\tau$-tilting pairs $B_{(U,Q)}^-(M,P)$ and $(M,P)$ correspond to the same  functorially finite torsion class. Hence, we have  $B_{(U,Q)}^-(M,P)=(M,P)$.
\end{proof}

In the rest part of this subsection, we show that the relative left Bongartz completions always exist for $\tau$-tilting finite algebras.

\begin{definition}[\cite{DIJ_2017}*{Definition 1.1}]
We say that $A$ is {\em $\tau$-tilting finite} if there are only finitely many basic $\tau$-tilting pairs in $\mod A$ up to isomorphisms, equivalently, $\mod A$ has only finitely many functorially finite torsion classes.
\end{definition}

\begin{theorem}[\cite{DIJ_2017}*{Theorem 1.2}
\label{thm:dij1.2}] 
The algebra $A$ is $\tau$-tilting finite if and only if every torsion class in $\mod A$ is functorially finite. 
\end{theorem}

\begin{corollary}\label{cor:dij}
Keep the notations in Proposition \ref{pro:finite}. Suppose that the algebra $A$ is $\tau$-tilting finite. Then $\Fac U\ast(\mathcal W\cap\Fac M)$ is a functorially finite torsion class in  $\ftors_{(U,Q)}(A)$ and the left Bongartz completion $B_{(U,Q)}^-(M,P)$ of $(U,Q)$ with respect to $(M,P)$ exists.
\end{corollary}
\begin{proof}
By Proposition \ref{pro:finite} (i), we know that $\Fac U\ast(\mathcal W\cap\Fac M)$ is a torsion class in  $\tors_{(U,Q)}(A)$. Since $A$ is $\tau$-tilting finite and by Theorem \ref{thm:dij1.2}, we know that $\Fac U\ast(\mathcal W\cap\Fac M)$ is a functorially finite torsion class in  $\ftors_{(U,Q)}(A)$. Thus  the left Bongartz completion $B_{(U,Q)}^-(M,P)$ of $(U,Q)$ with respect to $(M,P)$ exists.
\end{proof}

\subsection{Compatibility between relative left Bongartz completions and mutations}
In this subsection, we show that the relative left Bongartz completions have nice compatibility with mutations. 

Recall that a module $D\in\mod A$ is called a {\em brick}, if $\End_A(D)$ is a division ring.
\begin{lemma}\label{lem:brick}
Let $(M,P)$ and $(N,R)$ be two basic $\tau$-tilting pairs in $\mod A$. Then the following conditions are equivalent.

\begin{itemize}
    \item[(a)] $(N,R)$ is a left mutation of $(M,P)$.
    \item[(b)] $\Fac N\subsetneq \Fac M$ is a cover, i.e., there exists no torsion class $\mathcal T$ in $\mod A$ such that  $$\Fac N\subsetneq\mathcal T\subsetneq \Fac M.$$
    \item[(c)] There exists a unique brick $D$ in $N^\bot \cap\Fac M$ up to isomorphisms. In this case, we actually have $N^\bot \cap\Fac M=\Filt(D)$, where $\Filt(D)$ is the smallest full subcategory of $\mod A$ containing $D$ and closed under extensions.
\end{itemize}
\end{lemma}
\begin{proof}
The equivalence of (a) and (b) follows from \cite{adachi_iyama_reiten_2014}*{Theorem 2.33}. The equivalence of (b) and (c) can be found in many different references, for example, \cite{DIRRT}*{Theorem 3.3 (b)}, \cites{asai_2018,BCZ_2019}.
\end{proof}

\begin{theorem}\label{mainthm}
Let $(U,Q)$ be a basic $\tau$-rigid pair in $\mod A$ and $\mathcal W:=U^\bot\cap\prescript{\bot}{}({\tau U})\cap Q^\bot$ the associated wide subcategory.
Let  $(M,P)$ be a basic $\tau$-tilting pair in $\mod A$ and $(N,R)$ a left mutation of $(M,P)$. Denote by $D$ the unique brick in Lemma \ref{lem:brick}. Suppose that $\Fac M$ is a subcategory of $\prescript{\bot}{}({\tau U})\cap Q^\bot$.  Then the following statements hold.

\begin{itemize}
    \item [(i)]  If the brick $D$ does not belong to $U^\bot$, then $\mathcal W\cap\Fac M=\mathcal W\cap\Fac N$.

\item[(ii)] If the brick $D$ belongs to $U^\bot$, then $\mathcal W\cap\Fac M$ is a cover of $\mathcal W\cap\Fac N$ in the poset of torsion classes in  $\mathcal W$.

\item[(iii)] Either $B_{(U,Q)}^-(N,R)$ equals $B_{(U,Q)}^-(M,P)$ or $B_{(U,Q)}^-(N,R)$ is a left mutation of $B_{(U,Q)}^-(M,P)$. Namely, in the following commutative diagram,
$$\xymatrix{(N,R)\ar[d]^{B_{(U,Q)}^-}&&&(M,P)\ar[lll]_{\text{  left mutation}}\ar[d]^{B_{(U,Q)}^-}\\
B_{(U,Q)}^-(N,R)&&&B_{(U,Q)}^-(M,P)\ar[lll]_{\alpha}}$$
$\alpha$ is either an identity or a left mutation.
\end{itemize}
\end{theorem}

\begin{proof}

Since  $\Fac M$ is a subcategory of $\prescript{\bot}{}({\tau U})\cap Q^\bot$ and  $(N,R)$ is a left mutation of $(M,P)$, we have
$$\Fac N\subsetneq \Fac M\subseteq \prescript{\bot}{}({\tau U})\cap Q^\bot.$$
By  Proposition \ref{pro:finite} (iii), we have the equalities:
$$U^\bot\cap \Fac M=\mathcal W\cap \Fac M\;\;\;\text{and}\;\;\;U^\bot\cap \Fac N=\mathcal W\cap \Fac N.$$
By Proposition \ref{pro:finite} (v),  $B_{(U,Q)}^-(N,R)$ and $B_{(U,Q)}^-(M,P)$ can be defined.

(i) Since $\Fac N\subsetneq\Fac M$, we only need to show $\mathcal W\cap\Fac M\subseteq \mathcal W \cap\Fac N$. For any $X\in \mathcal W\cap\Fac M=U^\bot\cap \Fac M$, we will show $X\in\Fac N$ and thus $X\in \mathcal W\cap \Fac N$. Let us consider the canonical sequence of $X$ with respect to the torsion pair $(\Fac N, N^\bot)$ in $\mod A$
$$0\rightarrow X_t\rightarrow X\rightarrow X_f\rightarrow 0,$$
where $X_t\in\Fac N$ and $X_f\in N^\bot$. Since $\Fac M$ is closed under quotient modules and $X\in \Fac M$, we get  $X_f\in\Fac M$ and thus $X_f\in N^\bot \cap\Fac M=\Filt(D)$.

Now we show that $X_f=0$ using the condition $X\in U^\bot$. Because $U$ is $\tau$-rigid, we know that $U$ is Ext-projective in $\prescript{\bot}{}({\tau U})$. By
$$X_t\in  \Fac N\subsetneq  \Fac M\subseteq \prescript{\bot}{}({\tau U})\cap Q^\bot\subseteq \prescript{\bot}{}({\tau U}),$$
we get $\Ext_A^1(U,X_t)=0$. Applying the functor $\Hom_A(U,-)$ to $0\rightarrow X_t\rightarrow X\rightarrow X_f\rightarrow 0$, we get the following exact sequence:
$$0 \rightarrow \Hom_A(U,X_t)\rightarrow\Hom_A(U,X)\rightarrow\Hom_A(U,X_f)\rightarrow 0.$$
By $X\in U^\bot$, i.e., $\Hom_A(U,X)=0$, we get $\Hom_A(U,X_f)=0$. If $X_f\neq 0$, then it contains the brick $D$ as its submodule, by $X_f\in\Filt(D)$. This implies $\Hom_A(U, D)=0$, i.e., $D\in U^\bot$, which contradicts the assumption that $D$ does not belong to $U^\bot$. So we must have $X_f=0$ and thus $X\cong X_t\in\Fac N$. 

Hence, we have $X\in U^\bot\cap \Fac N=\mathcal W\cap \Fac N$ and thus $\mathcal W\cap\Fac M=\mathcal W\cap\Fac N$. 

(ii) By $\Fac N\subsetneq \Fac M$, we have $\mathcal W\cap\Fac N\subseteq \mathcal W\cap \Fac M$. Since $D$ is in $N^\bot \cap\Fac M$, we have $D\in \Fac M$ but $D\notin \Fac N$. By the assumption $D\in U^\bot$, we get $D\in U^\bot\cap \Fac M=\mathcal W\cap \Fac M$ but $D\notin U^\bot\cap\Fac N=\mathcal W\cap\Fac N$. Hence, $\mathcal W\cap\Fac N\subsetneq \mathcal W\cap \Fac M$. 

Let $\mathcal T$ be any torsion class in $\mathcal W$ such that 
$$\mathcal W\cap\Fac N\subsetneq \mathcal T\subseteq \mathcal W\cap\Fac M.$$
We will show that $\mathcal T=\mathcal W\cap\Fac M$ and thus $\mathcal W\cap\Fac M$ is a cover of $\mathcal W\cap\Fac N$ in the poset of torsion classes in $\mathcal W$.

We first show that the brick $D$ belongs to $\mathcal T$ and thus $\Filt(D)\subseteq\mathcal T$. Since $$U^\bot\cap\Fac N=\mathcal W\cap \Fac N\subsetneq \mathcal T,$$ there exists $Y\in\mathcal T$ but $Y\notin U^\bot\cap\Fac N$. Consider the canonical sequence of $Y$ with respect to the torsion pair $(\Fac N, N^\bot)$ in $\mod A$
$$0\rightarrow Y_t\rightarrow Y\rightarrow Y_f\rightarrow 0,$$
with $Y_t\in\Fac N$ and $Y_f\in N^\bot$. Since both $\mathcal T$ and $\Fac M$ are closed under quotient modules and by $Y\in\mathcal T\subseteq\Fac M$, we get $Y_f\in\mathcal T\cap \Fac M$ and thus $Y_f\in N^\bot\cap \Fac M=\Filt(D)$. 
By $Y\notin U^\bot\cap\Fac N$ and
$$Y\in\mathcal T\subseteq\mathcal W\cap\Fac M=U^\bot\cap\Fac M\subseteq U^\bot,$$
 we get $Y\notin \Fac N$. This implies $Y_f\neq 0$. Then by $Y_f\in\Filt(D)$, we know that $D$ is a quotient module of $Y_f$. Since $\mathcal T$ is closed under quotient modules and by $Y_f\in\mathcal T$, we get $D\in\mathcal T$. Since $\mathcal T$ is closed under extensions, we get $\Filt(D)\subseteq\mathcal T$.

Now we show $\mathcal W\cap \Fac M \subseteq\mathcal T$. For any $Z\in \mathcal W\cap\Fac M$, consider the canonical sequence of $Z$ with respect to the torsion pair $(\Fac N,N^\bot)$
$$0\rightarrow Z_t\rightarrow Z\rightarrow Z_f\rightarrow 0,$$
with $Z_t\in\Fac N$ and $Z_f\in N^\bot$. Since $\Fac M$ is closed under quotient modules and by $Z\in\Fac M$, we get $Z_f\in\Fac M$ and thus
$$Z_f\in N^\bot\cap \Fac M=\Filt(D)\subseteq\mathcal T.$$ 
Because $U^\bot$ is closed under submodules and by $Z\in\mathcal W\cap \Fac M=U^\bot\cap\Fac M\subseteq U^\bot$, we get $Z_t\in U^\bot$ and thus $Z_t\in U^\bot\cap \Fac N=\mathcal W\cap\Fac N\subsetneq\mathcal T$. Since $\mathcal T$ is closed under extensions and by $Z_t,Z_f\in\mathcal T$, we get $Z\in\mathcal T$ and thus $\mathcal W\cap\Fac M\subseteq\mathcal T$. So we have  $\mathcal T=\mathcal W\cap\Fac M$ and thus $\mathcal W\cap\Fac M$ is a cover of $\mathcal W\cap\Fac N$ in the poset of torsion classes in $\mathcal W$.

(iii) By (i) and (ii), we have that either $\mathcal W\cap\Fac M$ equals $\mathcal W\cap\Fac N$ or $\mathcal W\cap\Fac M$ is a cover of $\mathcal W\cap\Fac N$ in the poset of torsion classes in $\mathcal W$. Then by the order-preserving bijection in Theorem \ref{thm:Jasso}(iv), we get that either $\Fac U\ast(\mathcal W\cap\Fac M)$ equals $\Fac U\ast(\mathcal W\cap\Fac N)$ or 
$\Fac U\ast(\mathcal W\cap\Fac M)$ is a cover of  $\Fac U\ast(\mathcal W\cap\Fac N)$ in the poset of torsion classes in $\mod A$.

Notice that $B_{(U,Q)}^-(M,P)$ is the basic $\tau$-tilting pair corresponding to the functorially finite torsion class $\Fac U\ast(\mathcal W\cap\Fac M)$ 
in $\mod A$ and $B_{(U,Q)}^-(N,R)$ is the basic $\tau$-tilting pair corresponding to the functorially finite torsion class
$\Fac U\ast(\mathcal W\cap\Fac N)$
in $\mod A$. Then by Lemma \ref{lem:brick}, we get that  either $B_{(U,Q)}^-(N,R)$ equals $B_{(U,Q)}^-(M,P)$ or $B_{(U,Q)}^-(N,R)$ is a left mutation of $B_{(U,Q)}^-(M,P)$.\end{proof}

\begin{corollary}
Let $(M,P)$ be a basic $\tau$-tilting pair in $\mod A$ that can be obtained from the basic $\tau$-tilting pair $(A,0)$ by a sequence of mutations. Suppose that $(U,Q)$ is a common direct summand of $(M,P)$ and  $(A,0)$, then there exists a sequence of mutations going from $(A,0)$ to $(M,P)$ such that the basic $\tau$-rigid pair $(U,Q)$ remains the same during this process.
\end{corollary}
\begin{proof}
This result essentially dues to \cite{cao_2021}*{Theorem 5.12}, where it is proved in a different context.

Since $(U,Q)$ is a common direct summand of $(M,P)$ and  $(A,0)$, we know that $U$ is a basic projective module and $Q=0$. Thus $\prescript{\bot}{}({\tau U})\cap Q^\bot=\mod A$.
Since $(M,P)$ can be obtained from $(A,0)$ by a sequence of mutations, we can form the following diagram
$$\xymatrix{(M_0,P_0)&(M_1,P_1)\ar[l]_{\mu_{k_1}}&\ldots\ar[l]_{\;\;\;\;\;\mu_{k_2}}&(M_{m-1},P_{m-1})\ar[l]&(M_m,P_m)\ar[l]_{\;\;\;\;\;\;\;\;\mu_{k_m}}}$$
where $(M_0,P_0)=(M,P),\;(M_m,P_m)=(A,0)$ and each $\mu_{k_i}$ is a (left or right) mutation of $\tau$-tilting pairs.
Notice that if the mutation from $(M_{i},P_{i})$ to $(M_{i-1},P_{i-1})$ is a right mutation, then the mutation from $(M_{i-1},P_{i-1})$ to  $(M_{i},P_{i})$ is a left mutation. Since $\prescript{\bot}{}({\tau U})\cap Q^\bot=\mod A$, we know that $\Fac M_i\subseteq \prescript{\bot}{}({\tau U})\cap Q^\bot$ holds for any $i=0,1,\ldots,m$. Thus Theorem \ref{mainthm} (iii) can be applied to each mutation $\xymatrix{(M_{i-1},P_{i-1})\ar@{-}[r]^{\;\;\;\mu_{k_i}}&(M_i,P_i)}$ by choosing an orientation such that $\mu_{k_i}$ is a left mutation. So we have the following diagram
$$\xymatrix{(M_0,P_0)\ar[d]^{B_{(U,Q)}^-}&(M_1,P_1)\ar[d]^{B_{(U,Q)}^-}\ar@{-}[l]&\ldots\ar[d]^{B_{(U,Q)}^-}\ar@{-}[l]&(M_{m-1},P_{m-1})\ar[d]^{B_{(U,Q)}^-}\ar@{-}[l]&(M_m,P_m)\ar[d]^{B_{(U,Q)}^-}\ar@{-}[l]\\
(M_0^-,P_0^-)&(M_1^-,P_1^-)\ar@{-}[l]_{\alpha_1}&\ldots\ar@{-}[l]_{\;\;\;\;\;\;\alpha_2}&(M_{m-1}^-,P_{m-1}^-)\ar@{-}[l]_{}&(M_m^-,P_m^-)\ar@{-}[l]_{\;\;\;\;\;\;\;\alpha_m}}$$
where $(M_i^-,P_i^-)=B_{(U,Q)}^-(M_i,P_i)$ and each $\alpha_i$ is either an identity or a mutation of $\tau$-tilting pairs.

Since $(U,Q)$ is a common direct summand of $(M,P)=(M_0,P_0)$ and $(A,0)=(M_m,P_m)$ and by Corollary \ref{cortrivial}, we have $(M_0^-,P_0^-)=(M_0,P_0)=(M,P)$ and $(M_m^-,P_m^-)=(M_m,P_m)=(A,0)$. By $(M_i^-,P_i^-)=B_{(U,Q)}^-(M_i,P_i)$, we know that  each $(M_i^-,P_i^-)$ contains $(U,Q)$ as a direct summand. By deleting those $\alpha_i$ which are identities from
$$\xymatrix{
(M_0^-,P_0^-)&(M_1^-,P_1^-)\ar@{-}[l]_{\alpha_1}&\ldots\ar@{-}[l]_{\;\;\;\;\;\;\alpha_2}&(M_{m-1}^-,P_{m-1}^-)\ar@{-}[l]_{}&(M_m^-,P_m^-)\ar@{-}[l]_{\;\;\;\;\;\;\;\alpha_m}},$$ we obtain a sequence of mutations going from $(A,0)=(M_m^-,P_m^-)$ to $(M,P)=(M_0^-,P_0^-)$ such that the basic $\tau$-rigid pair $(U,Q)$ remains the same during this process.
\end{proof}

\begin{remark}
(i) In some sense, the corollary above is a representation-theoretic counterpart of Fomin-Zelevinsky's conjecture \cite{Fomin-Zelevinskt_04} on exchange graphs of cluster algebras stating that in a cluster algebra, the seeds whose clusters contain particular cluster variables form a connected subgraph of the exchange graph of this cluster algebra. This conjecture has been proved in \cite{cao-li-2020}.

(ii) The compatibility between completions and mutations is also very useful to study the non-leaving-face property of exchange graphs, cf., \cite{Fu-Geng-Liu_2022} for details. In the next subsection, we will give another application for such compatibility. 
\end{remark}

\subsection{Application to maximal green sequences}
We first recall the notion of maximal green sequences for torsion classes in $\mod A$. The existence of such sequences for finite dimensional Jacobian algebras \cite{DWZ_2008} has  important consequences in cluster algebras \cite{fz_2002}, cf., \cite{keller_demonet_2020}. In this subsection, we study the existence of maximal green sequences under $\tau$-tilting reduction for any finite dimensional algebra.

\begin{definition}[\cite{bst_2019}*{Definition 4.8}]
Let $\mathcal T$ be a torsion class in $\mod A$. A {\em maximal green sequence} for $\mathcal T$ is a finite sequence of torsion classes $$0=\mathcal T_0\subsetneq\mathcal T_1\subsetneq\ldots\subsetneq \mathcal T_{m-1}\subsetneq\mathcal T_m=\mathcal T$$
such that $\mathcal T_{i+1}$ is a cover of $\mathcal T_i$ for any $i=0,1,\ldots,m-1$.
\end{definition}
If $\mathcal T=\mod A$ has a maximal green sequence, we simply say that the algebra $A$ has a maximal green sequence.

\begin{proposition}[\cite{bst_2019}*{Proposition 4.9}\label{pro:bst4.9}]
Let  $\mathcal T$ be a torsion class in $\mod A$ and $$0=\mathcal T_0\subsetneq\mathcal T_1\subsetneq\ldots\subsetneq \mathcal T_{m-1}\subsetneq\mathcal T_m=\mathcal T$$ a maximal green sequence for $\mathcal T$ in $\mod A$. Then there exists a finite sequence of left mutations of basic $\tau$-tilting pairs
$$\xymatrix{(0,A)=(M_0,P_0)&(M_1,P_1)\ar[l]&\ldots\ar[l]&(M_{m-1},P_{m-1})\ar[l]&(M_m,P_m)\ar[l]}$$
such that $\Fac M_i=\mathcal T_i$ for $i=0,1,\ldots,m$.
\end{proposition}

\begin{theorem}\label{mainthm2}
Let $(U,Q)$ be a basic $\tau$-rigid pair in $\mod A$ and $(M,P)$ the absolute right Bongartz completion of $(U,Q)$, i.e., $\Fac M=\prescript{\bot}{}({\tau U})\cap Q^\bot$. Denote by $B=\End_A(M)$ and $A_{(U,Q)}=B/Be_UB$ the $\tau$-tilting reduction of $A$ at $(U,Q)$. If $\Fac M=\prescript{\bot}{}({\tau U})\cap Q^\bot$ has a maximal green sequence in $\mod A$, then the algebra $A_{(U,Q)}$ has a maximal green sequence.
\end{theorem}

\begin{proof}
Let $0=\mathcal T_0\subsetneq\mathcal T_1\subsetneq\ldots\subsetneq \mathcal T_{m-1}\subsetneq\mathcal T_m=\Fac M$ be a maximal green sequence for $\Fac M$ in $\mod A$.
By Proposition \ref{pro:bst4.9}, there exists a finite sequence of left mutations of basic $\tau$-tilting pairs
$$\xymatrix{(M_0,P_0)&(M_1,P_1)\ar[l]&\ldots\ar[l]&(M_{m-1},P_{m-1})\ar[l]&(M_m,P_m)\ar[l]}$$
such that $\Fac M_i=\mathcal T_i$ for $i=0,1,\ldots,m$. Notice that we have
$$(M_0,P_0)=(0,A)\;\;\;\text{and}\;\;\;(M_m,P_m)=(M,P).$$
By $\Fac M_i=\mathcal T_i\subseteq\mathcal T_m=\Fac M=\prescript{\bot}{}({\tau U})\cap Q^\bot$ and Proposition \ref{pro:finite} (v),  the left Bongartz completion $B_{(U,Q)}^-(M_i,P_i)$ of $(U,Q)$ with respect to $(M_i,P_i)$ can be defined, where $i=0,1,\ldots,m$. Thus we can form the following diagram
$$\xymatrix{(M_0,P_0)\ar[d]^{B_{(U,Q)}^-}&(M_1,P_1)\ar[d]^{B_{(U,Q)}^-}\ar[l]&\ldots\ar[d]^{B_{(U,Q)}^-}\ar[l]&(M_{m-1},P_{m-1})\ar[d]^{B_{(U,Q)}^-}\ar[l]&(M_m,P_m)\ar[d]^{B_{(U,Q)}^-}\ar[l]\\
(M_0^-,P_0^-)&(M_1^-,P_1^-)\ar[l]_{\alpha_1}&\ldots\ar[l]_{\;\;\;\;\;\;\alpha_2}&(M_{m-1}^-,P_{m-1}^-)\ar[l]_{}&(M_m^-,P_m^-)\ar[l]_{\;\;\;\;\;\;\;\alpha_m}}$$
where $(M_i^-,P_i^-)=B_{(U,Q)}^-(M_i,P_i)$ for $i=0,1,\ldots,m$. By Proposition \ref{pro:finite} (v), we have
\begin{eqnarray}
\Fac M_0^-&=&\Fac U\ast  \Fac M_0=\Fac U,\nonumber\\
\Fac M_m^-&=&\Fac U\ast \Fac M_m=\prescript{\bot}{}({\tau U})\cap Q^\bot.\nonumber
\end{eqnarray}
By Theorem \ref{mainthm}, we know that each $\alpha_i$ in the above diagram is either an identity or a left mutation. So $\xymatrix{
(M_0^-,P_0^-)&(M_1^-,P_1^-)\ar[l]_{\alpha_1}&\ldots\ar[l]_{\;\;\;\;\;\;\alpha_2}&(M_{m-1}^-,P_{m-1}^-)\ar[l]_{}&(M_m^-,P_m^-)\ar[l]_{\;\;\;\;\;\;\;\alpha_m}}$ induces a chain of torsion classes
$$\Fac U=\Fac M_0^-\subseteq\Fac M_1^-\subseteq\ldots\subseteq\Fac M_{m-1}^-\subseteq\Fac M_m^-=\prescript{\bot}{}({\tau U})\cap Q^\bot$$
in $\ftors_{(U,Q)}(A)$. After removing the redundant torsion classes in the above chain, we get a new chain
$$\Fac U=\Fac M_0^-\subsetneq\Fac M_{k_1}^-\subsetneq\ldots\subsetneq\Fac M_{k_{\ell}}^-\subsetneq\Fac M_m^-=\prescript{\bot}{}({\tau U})\cap Q^\bot$$
in $\ftors_{(U,Q)}(A)$ such that each inclusion above is a cover of torsion classes.
This new chain will give us a maximal green sequence of the algebra $A_{(U,Q)}$, by the order-preserving bijections in Theorem \ref{thm:Jasso} (iv) and (iii).
\end{proof}

\begin{remark}
In the above theorem, if we take $(U,Q)$ a basic $\tau$-rigid pair in $\add (A,0)$, then $\prescript{\bot}{}({\tau U})\cap Q^\bot=\mod A,\; (M,P)=(A,0),\; B=\End_A(M)=A$ and $A_{(U,Q)}=A/Ae_UA$. In this case, the above theorem is just the  representation-theoretic version
 of  Muller’s theorem \cite{Muller16}
stating that full subquivers of a quiver inherit maximal green sequences.
\end{remark}

Let us give a concrete example for the case $(U,Q)\in\add (A,0)$. 
\begin{example}
Let $A$ be the $K$-algebra given by the following quiver
$$
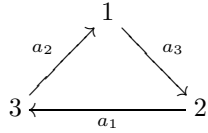
\xymatrix{&1\ar[rd]^{a_3}&\\
3\ar[ru]^{a_2}&&2\ar[ll]^{a_1}}$$
with relations: $a_1a_2=0, \;a_2a_3=0,\; a_3a_1=0$. We denote by $S_i\in\mod A$ the simple module and $P_i\in\mod A$ the indecomposable projective module corresponding to the vertex $i$. Consider the following sequence of left mutations of $\tau$-tilting pairs in $\mod A$
$$\xymatrix{(0,A)&(S_3,P_2\oplus P_1)\ar[l]&(S_3\oplus P_3,P_2)\ar[l]&(S_3\oplus P_2\oplus P_3,0)\ar[l]&(A,0)\ar[l]}.$$
This sequence of left mutations gives a maximal green sequence of the algebra $A$
$$0\subsetneq \Fac S_3 \subsetneq \Fac (S_3\oplus P_3)\subsetneq\Fac (S_3\oplus P_2\oplus P_3)\subsetneq\Fac A=\mod A.$$

Keep the notations in Theorem \ref{mainthm2}.
We take $(U,Q)=(P_1,0)$. Then  $\prescript{\bot}{}({\tau U})\cap Q^\bot=\mod A,\; (M,P)=(A,0),\; B=\End_A(M)=A$. In this case, $A_{(U,Q)}=A/Ae_UA$ is just the path algebra of the quiver  $$\xymatrix{3^\prime&2^\prime\ar[l]}.$$ 
Applying $B_{(U,Q)}^-=B_{(P_1,0)}^-$ to the above sequence of left mutations of $\tau$-tilting pairs, we obtain the following chain of $\tau$-tilting pairs containing $(P_1,0)$ as a direct summand
$$\xymatrix{(P_1\oplus S_1, P_3)&(P_1\oplus S_1\oplus P_3,0)\ar[l]&(P_1\oplus S_1\oplus P_3,0)\ar[l]&(A,0)\ar[l]&(A,0)\ar[l]}.$$
Then it induces a chain of torsion classes in $\ftors_{(P_1,0)}(A)$
$$\Fac (P_1\oplus S_1)\subsetneq\Fac(P_1\oplus S_1\oplus P_3) \subsetneq \mod A.$$
Using the order-preserving bijections in Theorem \ref{thm:Jasso} (iv) and (iii), we get a new chain
$$0\subsetneq \Fac S_{3^\prime}\subsetneq \mod A_{(P_1,0)}$$
of torsion classes in $\mod A_{(P_1,0)}$, where $S_{3^\prime}\in\mod A_{(P_1,0)}$ is the simple module corresponding to the vertex $3^\prime$. It is easy to see that $0\subsetneq \Fac S_{3^\prime}\subsetneq \mod A_{(P_1,0)}$ is a maximal green sequence of the algebra $A_{(U,Q)}=A_{(P_1,0)}$.
\end{example}

\section{Relative left Bongartz completions in silting theory}\label{sec:4}
In this section, we fix a $K$-linear, Krull-Schmidt, Hom-finite triangulated category $\mathcal D$ with a silting object.

\subsection{Relative left Bongartz completions in silting theory}

We associate a basic object $B_U^-(T)$ to any pair $(U,T)$ of basic objects in $\mathcal D$ as follows:
take a triangle 
$$\xymatrix{T[-1]\ar[r]^f&U^\prime\ar[r]&X_U\ar[r]&T}$$ 
with a minimal left $\add U$-approximation $f:T[-1]\rightarrow U^\prime$. Denote by $B_U^-(T):=(U\oplus X_U)^\flat$ the basic object corresponding to $U\oplus X_U$.

\begin{proposition}\label{pro:common}
Let $M=\bigoplus_{i=1}^n M_i$ be a basic silting object in $\mathcal D$ and $N$ a left mutation of $M$. Let $U$ be a basic object in $\mathcal D$. Then $B_U^-(M)$ and $B_U^-(N)$ have at least $n-1$ common indecomposable direct summands. 
\end{proposition}

\begin{proof}

Say $N$ is the left mutation of $M$ at $M_k$. Then by definition, there exists an indecomposable object $M_k^\prime$ such that $N=M_k^\prime\bigoplus(\bigoplus_{j\neq k}M_j)$. So  $\bigoplus_{j\neq k}M_j$ is a common direct summand of $M$ and $N$.

For each $j\neq k$, we take a triangle $$\xymatrix{M_j[-1]\ar[r]^{f_j}&U_j^\prime\ar[r]&X_j\ar[r]&M_j}$$ 
with a minimal left $\add U$-approximation $f_j:M_j[-1]\rightarrow U_j^\prime$. By definitions of $B_U^-(M)$ and $B_U^-(N)$, we know that
$L:=(U\bigoplus (\bigoplus_{j\neq k} X_j))^\flat$ is a common direct summand of $B_U^-(M)$ and $B_U^-(N)$.

Now we show that the basic object $L$ has at least $n-1$ indecomposable direct summands. Let  $$L=(U\bigoplus (\bigoplus_{j\neq k} X_j))^\flat=\bigoplus_{i=1}^sL_i$$ be a direct sum decomposition such that each $L_i$ is indecomposable. By the triangles above, we have the following equality in the Grothendieck group $\K_0(\mathcal D)$ of $\mathcal D$
$$[M_j]=[X_j]-[U_j^\prime]$$
for any $j\neq k$. Notice that $X_j, U_j^\prime$ belong to $\add L$. Hence, we get that $[X_j], [U_j^\prime]$ are in the $\mathbb Z$-span of $\{[L_1],\ldots,[L_s]\}$. So $[M_j]=[X_j]-[U_j^\prime]$ is in the $\mathbb Z$-span of $\{[L_1],\ldots,[L_s]\}$ for any $j\neq k$.

Since $M$ is a basic silting object in $\mathcal D$ and by Theorem \ref{thm:k0}, we know that  $$[M_1],\ldots,[M_{k-1}],[M_{k+1}],\ldots,[M_n]$$ are linearly independent over the field $\mathbb Q$. Since they can be spanned by $[L_1],\ldots,[L_s]$ over $\mathbb Q$, we must have $s\geq n-1$. Thus $L$ has at least $n-1$ indecomposable direct summands. This implies that $B_U^-(M)$ and $B_U^-(N)$ have at least $n-1$ common indecomposable direct summands.
\end{proof}

\begin{definition}[Relative left Bongartz completion]\label{def:insilting}
Let $S$ be a basic silting object in $\mathcal D$ and $U$  a basic presilting object in $\mathcal C=S\ast S[1]$. Let $T$ be a basic silting object in $\mathcal C$ such that $B_U^-(T)$ is still a basic silting object in $\mathcal C$. We call the basic silting object $B_U^-(T)$ the {\em left Bongartz completion} of $U$ with respect to $T$.
\end{definition}

In the following proposition, we will give a sufficient condition such that  $B_U^-(T)$ is a basic silting object in $\mathcal C=S\ast S[1]$.

\begin{proposition}\label{pro-def}
Let $S$ be a basic silting object in $\mathcal D$ and denote by $\mathcal C=S\ast S[1]$. Let $U$ be a basic presilting object in $\mathcal C$ and $T$ a basic silting object in $\mathcal C$.  Suppose that $\Hom_{\mathcal D}(U,T[i])=0$ holds for any positive integer $i$ (thanks to Proposition \ref{pro:rigid} below, this is equivalent to requiring $\Hom_{\mathcal D}(U,T[1])=0$). Then $B_U^-(T)$ is a basic silting object in $\mathcal C$.
\end{proposition}
\begin{proof}
The proof for the case $U\in \add(S)$ is the same as the proof \cite{cao_2021}*{Theorem 6.4}. Notice that if $U\in\add(S)$, then the assumption that $\Hom_{\mathcal D}(U,T[i])=0$
is automatically valid for any positive integer $i$. For the case $U\notin\add(S)$, the proof in \cite{cao_2021} still works under the assumption that $\Hom_{\mathcal D}(U,T[i])=0$ holds for any positive integer $i$. For the convenience of readers, we sketch the proof here.

Notice that the subcategory $\mathcal C=S\ast S[1]$ is closed under extensions in $\mathcal D$, by \cite{cao_2021}*{Corollary 3.9} and it is closed under direct summands, by \cite{Iyama-Yoshino_2008}*{Proposition
2.1(1)}.

Take a triangle 
\begin{eqnarray}\label{eqn:3gon}
\xymatrix{T[-1]\ar[r]^f&U^\prime\ar[r]&X_U\ar[r]&T}
\end{eqnarray}
with a minimal left $\add U$-approximation $f:T[-1]\rightarrow U^\prime$. By definition,  $B_U^-(T)=(U\oplus X_U)^\flat$.

We first show that $U\oplus X_U$ is an rigid object in $\mathcal C$, that is, $U\oplus X_U$ belongs to $\mathcal C$ and $$\Hom_{\mathcal D}(U\oplus X_U,U[1]\oplus X_U[1])=0.$$ By $U^\prime\in\add U\subseteq \mathcal C$, $T\in \mathcal C$ and the fact that 
 $X_U$ is an extension of $U^\prime$ and $T$, we get $X_U\in \mathcal C$ and thus $U\oplus X_U\in \mathcal C$. Since $U$ is presilting, we have $\Hom_{\mathcal D}(U,U[1])=0$. Hence, in order to show that $U\oplus X_U$ is an rigid object in $\mathcal C$, we still need to show 
$$\Hom_{\mathcal D}(X_U,U[1])=0\;\;\;\text{and}\;\;\;\Hom_{\mathcal D}(U\oplus X_U,X_U[1])=0.$$

Now we show $\Hom_{\mathcal D}(X_U, U[1])=0$.
Applying the functor $\Hom_{\mathcal D}(-, U)$ to the triangle ($\ref{eqn:3gon}$), we get the following exact sequence:
$$\xymatrix{\Hom_{\mathcal D}(U^\prime, U)\ar[r]&
\Hom_{\mathcal D}(T[-1], U)\ar[r]&
\Hom_{\mathcal D}(X_U[-1], U)\ar[r]&
\Hom_{\mathcal D}(U^\prime[-1], U)=0}.$$
Since $f$ is a left $\add U$-approximation of $T[-1]$, we get that the morphism  $$\xymatrix{\Hom_{\mathcal D}(U^\prime, U)\ar[r]&
\Hom_{\mathcal D}(T[-1], U)}$$ is surjective and thus the morphism $\xymatrix{\Hom_{\mathcal D}(X_U[-1], U)\ar[r]&
\Hom_{\mathcal D}(U^\prime[-1], U)=0}$ is injective. Hence, 
$$\Hom_{\mathcal D}(X_U, U[1])\cong\Hom_{\mathcal D}(X_U[-1], U)=0.$$

Now we show $\Hom_{\mathcal D}(X_U,T[1])=0$. Since $T$ is silting, we have $\Hom_{\mathcal D}(T,T[1])=0$. By the assumption that  $\Hom_{\mathcal D}(U,T[i])=0$ holds for any positive integer $i$, we have  $\Hom_{\mathcal D}(U,T[1])=0$. Since $X_U$ is an extension of $U^\prime\in\add U$ and $T$, we get  $\Hom_{\mathcal D}(X_U,T[1])=0$.

Notice that we have already known $\Hom_{\mathcal D}(U\oplus X_U,U[1])=0$ and $\Hom_{\mathcal D}(U\oplus X_U,T[1])=0$. Since $X_U$ is an extension of $U^\prime$ and $T$, we get $\Hom_{\mathcal D}(U\oplus X_U,X_U[1])=0$.

In summary, we get that $U\oplus X_U$ is an rigid object in $\mathcal C=S\ast S[1]$. Then by Proposition \ref{pro:rigid} below, we know that $U\oplus X_U$ is a presilting object in $\mathcal C$. Since $T$ is silting in $\mathcal C$ and by the triangle ($\ref{eqn:3gon}$), we know that the presilting object $U\oplus X_U$ is a silting object in  $\mathcal C$. Since $\mathcal C$ is closed under direct summands, we get that $B_U^-(T)=(U\oplus X_U)^\flat$ is a basic silting object in  $\mathcal C$.
\end{proof}

\begin{proposition}[\cite{cao_2021}*{Lemma 6.3}]  \label{pro:rigid}
Let $S$ be a basic silting object in $\mathcal D$, and let $U,T$ be any two objects in $\mathcal C=S\ast S[1]$. Then $\Hom_{\mathcal D}(U,T[i])=0$ for any $i\geq 2$. In particular, an object $U$ in $\mathcal C$ is presilting if and only if it is rigid, i.e., $\Hom_{\mathcal D}(U,U[1])=0$.
\end{proposition}

\begin{theorem}\label{mainthm3}
Let $S=\bigoplus_{i=1}^nS_i$ be a basic silting object in $\mathcal D$ and $U$ a basic presilting object in $\mathcal C=S\ast S[1]$. Let $M=\bigoplus_{i=1}^n M_i$ be a basic silting object in $\mathcal C$ and $N$ a left mutation of $M$ with $N$ still in $\mathcal C$. Suppose that $\Hom_{\mathcal D}(U,M[i])=0$ holds for any positive integer $i$ (thanks to Proposition \ref{pro:rigid}, this is equivalent to requiring $\Hom_{\mathcal D}(U,M[1])=0$). Then the following statements hold.
\begin{itemize}
\item[(i)] $\Hom_{\mathcal D}(U,N[i])=0$ holds for any positive integer $i$.
    \item[(ii)]  $B_U^-(M)$ and $B_U^-(N)$ are two basic silting objects in $\mathcal C$.
    \item[(iii)] Either  $B_U^-(N)$ equals $B_U^-(M)$ or  $B_U^-(N)$ is a left mutation of $B_U^-(M)$. Namely, in the following commutative diagram,
    $$\xymatrix{N\ar[d]^{B_U^-}&&&M\ar[d]^{B_U^-}\ar[lll]_{\text{left mutation}}\\
    B_U^-(N)&&&B_U^-(M)\ar[lll]_{\alpha}}$$
    $\alpha$ is either an identity or a left mutation.
\end{itemize}
\end{theorem}
\begin{proof}
(i) Say $N$ is a left mutation of $M=\bigoplus_{i=1}^n M_i$ at $M_k$ and denote by $L=\bigoplus_{j\neq k}M_j$. Take a triangle
\begin{eqnarray}\label{eqn:triZ}
\xymatrix{M_k\ar[r]^f&L^\prime \ar[r]&Z\ar[r]&M_k[1]}
\end{eqnarray}
with a minimal left $\add L$-approximation $f:M_k\rightarrow L^\prime$ of $M_k$. By the definition of left mutation, we know that $N=(L\oplus Z)^\flat$.

Fix a positive integer $i$, we need to show $\Hom_{\mathcal D}(U,N[i])=0$. It suffices to show $\Hom_{\mathcal D}(U,L[i])=0$ and $\Hom_{\mathcal D}(U,Z[i])=0$. By $\Hom_{\mathcal D}(U,M[i])=0=\Hom_{\mathcal D}(U,M[1+i])$ and  $L,L^\prime,M_k\in \add M$, we get that $\Hom_{\mathcal D}(U,L[i])=0$, $\Hom_{\mathcal D}(U,L^\prime[i])=0$ and $\Hom_{\mathcal D}(U,M_k[1+i])=0$.
By the triangle ($\ref{eqn:triZ}$), we know that $Z$ is an extension of $L^\prime$ and $M_k[1]$. Using this fact, we get  $\Hom_{\mathcal D}(U,Z[i])=0$. Hence, $\Hom_{\mathcal D}(U,N[i])=0$, where $i$ is any fixed positive integer.

(ii)  Since $\Hom_{\mathcal D}(U,M[i])=0$ and $\Hom_{\mathcal D}(U,N[i])=0$ hold for any positive integer $i$ and by Proposition \ref{pro-def}, we know that $B_U^-(M)$ and $B_U^-(N)$ are basic silting objects in $\mathcal C$.

(iii) 
By (ii),  we know that $B_U^-(M)$ and $B_U^-(N)$ are basic silting objects in $\mathcal C$. Thus they have $n$ indecomposable direct summands, by Theorem \ref{thm:k0}.

By Proposition \ref{pro:common}, 
$B_U^-(M)$ and $B_U^-(N)$ have at least $n-1$ common indecomposable direct summands.
So we get that  $B_U^-(M)$ and $B_U^-(N)$ have either $n-1$ common indecomposable direct summands or $n$ common indecomposable direct summands.

If $B_U^-(M)$ and $B_U^-(N)$ have $n$ common indecomposable direct summands, then they are the same.

If $B_U^-(M)$ and $B_U^-(N)$ have $n-1$ common indecomposable direct summands, then they differ by a mutation, by Theorem \ref{thm:OIY14} and Theorem \ref{thmair}. In this case, we still need to show that the mutation from $B_U^-(M)$ to $B_U^-(N)$ is a left mutation. It suffices to show $B_U^-(N)\leq B_U^-(M)$, by Theorem \ref{thm:AIcover} (ii).

By definition, $B_U^-(N)\leq B_U^-(M)$ if and only if $\Hom_{\mathcal D}(B_U^-(M),B_U^-(N)[i])=0$ for any positive integer $i$. Let us first recall the constructions of $B_U^-(M)$ and $B_U^-(N)$. Take two triangles
\begin{eqnarray}
\xymatrix{M[-1]\ar[r]^{f}&U_M\ar[r]&X\ar[r]&M },\nonumber\\
\xymatrix{N[-1]\ar[r]^{g}&U_N\ar[r]&Y\ar[r]&N},\nonumber
\end{eqnarray}
where $f$ is a minimal left $\add U$-approximation of $M[-1]$ and $g$ is a minimal left $\add U$-approximation of $N[-1]$. Then $B_U^-(M)=(U\oplus X)^\flat$ and $B_U^-(N)=(U\oplus Y)^\flat$.

Fix a positive integer $i$, we need to show  $\Hom_{\mathcal D}(B_U^-(M),B_U^-(N)[i])=0$, equivalently, to show  $\Hom_{\mathcal D}(U\oplus X,(U\oplus Y)[i])=0$. Since $B_U^-(M)=(U\oplus X)^\flat$ and $B_U^-(N)=(U\oplus Y)^\flat$ are silting objects in $\mathcal D$, we have
$$\Hom_{\mathcal D}(U,U[j])=0,\;\;\Hom_{\mathcal D}(U,Y[j])=0,\;\;\Hom_{\mathcal D}(X,U[j])=0,\;\;\Hom_{\mathcal D}(X,X[j])=0$$
for any $j>0$. In particular, these equalities hold for $j=i$. It remains to show $\Hom_{\mathcal D}(X,Y[i])=0$.

We first show that $\Hom_{\mathcal D}(X,M[i])=0=\Hom_{\mathcal D}(X,M[1+i])$. This actually follows from the fact that $M$ is an extension of $X$ and $U_M[1]\in\add(U[1])$ and we have already known $$\Hom_{\mathcal D}(X,X[j])=0=\Hom_{\mathcal D}(X,U[1+j])$$ for any positive integer $j$.

Now we show that $\Hom_{\mathcal D}(X,N[i])=0$. 
Since $N=(L\oplus Z)^\flat$, it suffices to show 
$\Hom_{\mathcal D}(X,L[i])=0$ and $\Hom_{\mathcal D}(X,Z[i])=0$. By $\Hom_{\mathcal D}(X,M[i])=0$ and $L,L'\in\add M$, we have  $$\Hom_{\mathcal D}(X,L[i])=0=\Hom_{\mathcal D}(X,L'[i]).$$ Since $Z$ is an extension of $L^\prime$ and $M_k[1]\in\add(M[1])$, we get $\Hom_{\mathcal D}(X,Z[i])=0$. Thus $\Hom_{\mathcal D}(X,N[i])=0$.

Since  $Y$ is an extension of $U_N\in\add U$ and $N$, we get $\Hom_{\mathcal D}(X,Y[i])=0$.
In summary, we have proved  $\Hom_{\mathcal D}(U\oplus X,(U\oplus Y)[i])=0$ for any fixed positive integer $i$. This implies $$\Hom_{\mathcal D}(B_U^-(M),B_U^-(N)[i])=0$$ for any $i>0$ and $B_U^-(N)\leq B_U^-(M)$. So the mutation from $B_U^-(M)$ to $B_U^-(N)$ is a left mutation.
\end{proof}

\subsection{Compatibility of left completions in silting theory and $\tau$-tilting theory}

Let $S$ be a basic silting object in $\mathcal D$ and denote by $\mathcal C=S\ast S[1]$. By \cite{cao_2021}*{Corollary 3.9}, $\mathcal C$ is  closed under extensions in $\mathcal D$. Hence, it inherits some structure from the triangulated structure of $\mathcal D$: It is an extriangulated category in the sense of \cite{NP_2019}.

Let $\underline{\mathcal C}$ be the quotient of $\mathcal C$ by the ideal of $\mathcal C$ generated by the morphisms $f$ such that $\Hom_{\mathcal D}(f,X[1])=0$ for any $X\in\mathcal C$. We call  $\underline{\mathcal C}$ the {\em stable category} of $\mathcal C$. Dually, we can define the {\em costable category} $\overline{\mathcal C}$ to be the quotient of $\mathcal C$ by the ideal of $\mathcal C$ generated by the morphisms $g$ such that $\Hom_{\mathcal D}(Y[-1],g)=0$ for any $Y\in\mathcal C$.

Notice that the extriangulated category $\mathcal C$ has enough projectives and enough injectives (in the sense of \cite{NP_2019}*{Definition 3.25}) with  $\proj(\mathcal C)=\add S$ and $\inj(\mathcal C)=\add(S[1])$. Then by \cite{INP_2018}*{Remark 1.23}, we have
$$\underline{\mathcal C}=\frac{\mathcal C}{[S]}\;\;\;\text{and}\;\;\; \overline{\mathcal C}=\frac{\mathcal C}{[S[1]]},$$
where $[S]$ and $[S[1]]$ are the ideals of $\mathcal C$ consisting of morphisms factoring through  $\add S$ and $\add (S[1])$, respectively.

\begin{definition}[Auslander–Reiten–Serre duality, \cite{INP_2018}*{Definition 3.4}]
An {\em Auslander–Reiten–Serre (ARS) duality} of $\mathcal C=S\ast S[1]$ is a pair $(\tau_S,\eta)$ such that $\tau_S:\underline{\mathcal C}\rightarrow \overline{\mathcal C}$ is an equivalence 
of additive categories and $\eta$ is a natural isomorphism
\begin{eqnarray}\label{eqn:tau1}
\eta_{X,Y}:\Hom_{\underline{\mathcal C}}(X,Y)\cong\mathbb D\Hom_{\mathcal C}(Y,\tau_SX[1]),
\end{eqnarray}
where $\mathbb D$ is the $K$-dual and $X,Y\in\mathcal C$.
\end{definition}

\begin{remark}\label{rmk:ars}
Since  $\tau_S:\underline{\mathcal C}\rightarrow \overline{\mathcal C}$ is an equivalence, the equality ($\ref{eqn:tau1}$) is equivalent to $$\Hom_{\overline{\mathcal C}}(\tau_SX,\tau_SY)\cong\mathbb D\Hom_{\mathcal C}(Y,\tau_SX[1]),$$i.e., $\Hom_{\overline{\mathcal C}}(X,\tau_SY)\cong\mathbb D\Hom_{\mathcal C}(Y,X[1]),$ for any $X,Y\in\mathcal C$.
\end{remark}

{\em In the rest part of this section, we further assume that the extriangulated category $\mathcal C=S\ast S[1]$ has an ARS-duality $(\tau_S,\eta)$}. For example, we can take $\mathcal D=\K^b(\proj A)$ the bounded homotopy category of finitely generated projective modules $\proj A$ for a finite dimensional algebra $A$ and take $S=A$ the stalk complex concentrated in degree $0$, cf., \cite{INP_2018}*{Example 5.18}.

The following decomposition will be used frequently without further explanation. For any object $U\in\mathcal C=S\ast S[1]$, we write 
\begin{eqnarray}\label{eqn:notation}
    U=X_U\oplus Y_U[1]
\end{eqnarray}
 for a decomposition such that $Y_U[1]$ is a maximal direct summand of $U$ that belongs to $\add (S[1])$. Such a decomposition is unique up to isomorphisms.

{\em Setting:} Let $A=\End_{\mathcal D}(S)$ be the endomorphism algebra of $S$. By Theorem \ref{thm:OIY14}, we have a functor  $F=\Hom_{\mathcal D}(S,-):\mathcal C=S\ast S[1]\rightarrow \mod A$, which induces an equivalence $\overline{(-)}:\overline{\mathcal C}\rightarrow\mod A
$ and  a bijection from presilting objects in $\mathcal C$ to $\tau$-rigid pairs in $\mod A$ via
    $$U=X_U\oplus Y_U[1]\mapsto (F(X_U),F(Y_U))=(\overline X_U,\overline Y_U).$$

\begin{proposition}\label{pro:compare1}
Keep the above setting. Let $U=X_U\oplus Y_U[1]$ be a basic presilting object in $\mathcal C=S\ast S[1]$ and $T=X_T\oplus Y_T[1]$ a basic silting object in $\mathcal C$. Then the following conditions are equivalent.
\begin{itemize}
    \item [(a)] $\Hom_{\mathcal D}(U,T[i])=0$ for any positive integer $i$.
    \item[(b)]  $\Hom_{\mathcal D}(U,T[1])=0$.
    \item[(c)] $\Hom_{\mathcal D}(U,X_T[1])=0$.
    \item[(d)] $\Hom_{\overline{\mathcal C}}(Y_U,X_T)=0$ and $\Hom_{\overline{\mathcal C}}(X_T,\tau_{S}X_U)=0$.
    \item [(e)] $\Fac \overline X_T$ is a subcategory of  $\prescript{\bot}{}({\tau \overline X_U})\cap \overline Y_U^\bot$.
\end{itemize}

\end{proposition}
\begin{proof}
``(a)$\Longleftrightarrow$(b)": This follows from Proposition \ref{pro:rigid}.

``(b)$\Longleftrightarrow$(c)": Since $T=X_T\oplus Y_T[1]$, we have $\Hom_{\mathcal D}(U,T[1])\cong \Hom_{\mathcal D}(U, X_T[1])\oplus \Hom_{\mathcal D}(U,Y_T[2])$. Notice that for any object $Z\in \mathcal C=S\ast S[1]$, we always have   $\Hom_{\mathcal D}(Z,S[2])=0$. By $Y_T\in \add S$, we get that $\Hom_{\mathcal D}(U,Y_T[2])$ is always $0$. Hence, (b) and (c) are equivalent.

``(c)$\Longleftrightarrow$(d)": By $U=X_U\oplus Y_U[1]$, we know that
$\Hom_{\mathcal D}(U,X_T[1])\cong\Hom_{\mathcal D}(X_U,X_T[1])\oplus \Hom_{\mathcal D}(Y_U,X_T)$. In order to show the equivalence of (c) and (d), it suffices to show $\Hom_{\overline{\mathcal C}}(Y_U,X_T)\cong\Hom_{\mathcal D}(Y_U,X_T)$ and $\Hom_{\overline{\mathcal C}}(X_T,\tau_{S}X_U) \cong \mathbb D\Hom_{\mathcal D}(X_U, X_T[1])$.

Since $Y_U$ is in $\add S$ and by $\Hom_{\mathcal D}(S,S[1])=0$, we know that any morphism in $\Hom_{\mathcal C}(Y_U,X_T)$ that can factor through $\add (S[1])$ must be zero. Thus we get 
$$\Hom_{\overline{\mathcal C}}(Y_U,X_T)\cong \Hom_{\mathcal C}(Y_U,X_T)=\Hom_{\mathcal D}(Y_U,X_T).$$

By ARS duality in Remark \ref{rmk:ars}, we have 
$$\Hom_{\overline{\mathcal C}}(X_T,\tau_{S}X_U)\cong \mathbb D\Hom_{\mathcal C}(X_U, X_T[1])=\mathbb D\Hom_{\mathcal D}(X_U, X_T[1]).$$

Hence, (c) and (d) are equivalent.

``(d)$\Longleftrightarrow$(e)": 
Since $X_U$ has no direct summands in $\add(S[1])$ and by \cite{Iyama-Yoshino_2008}*{Corollary 6.5 (3)(iii)}, we know that $\overline{\tau_S X_U}=\tau \overline X_U$.  Since the functor $\overline{(-)}:\overline{\mathcal C}\rightarrow\mod A
$ is an equivalence, we know that $$\Hom_{\overline{\mathcal C}}(Y_U,X_T)=0=\Hom_{\overline{\mathcal C}}(X_T,\tau_{S}X_U)$$ holds if and only if $$\Hom_A(\overline Y_U, \overline X_T)=0=\Hom_A(\overline X_T, \overline {\tau_SX_U})=\Hom_A(\overline X_T, \tau \overline X_U)$$ holds.
It is easy to see $\Hom_A(\overline X_T, \tau \overline X_U)=0$ if and only if $\Fac \overline X_T\subseteq \prescript{\bot}{}({\tau \overline X_U})$. Since $\overline Y_U$ is a projective $A$-module, $\Hom_A(\overline Y_U, \overline X_T)=0$ if and only if $\Fac \overline X_T\subseteq \overline Y_U^\bot$.
Hence, (d) and (e) are equivalent. 
\end{proof}

\begin{theorem}\label{mainthm4}
Keep the above setting. Let $U=X_U\oplus Y_U[1]$ be a basic presilting object in $\mathcal C=S\ast S[1]$ and $T=X_T\oplus Y_T[1]$ a basic silting object in $\mathcal C$. Suppose that $\Hom_{\mathcal D}(U,T[i])=0$ holds for any positive integer $i$. Then the following statements hold.
\begin{itemize}
    \item [(i)] $B_U^-(T)$ is a basic silting object in $\mathcal C$. In particular, it corresponds to a basic $\tau$-tilting pair $(\overline X_M,\overline Y_M)$ in $\mod A$, where $M:=B_U^-(T)=X_M\oplus Y_M[1]$.
    \item[(ii)] The left Bongartz completion 
    $B_{(\overline X_U,\overline Y_U)}^-(\overline X_T,\overline Y_T)$ of the $\tau$-rigid pair $(\overline X_U,\overline Y_U)$ with respect to the $\tau$-tilting pair $(\overline X_T,\overline Y_T)$ exists in $\mod A$.
    \item[(iii)] The two basic $\tau$-tilting pairs
$(\overline X_M,\overline Y_M)$ and $B_{(\overline X_U,\overline Y_U)}^-(\overline X_T,\overline Y_T)$ are the same up to isomorphisms. Namely, we have the following commutative diagram.
$$\xymatrix{T=X_T\oplus Y_T[1]\ar[d]^{B_U^-}\ar[rrr]^{\overline{(-)}}&&&(\overline X_T,\overline Y_T)\ar[d]^{B_{(\overline X_U,\overline Y_U)}^-}\\
M=X_M\oplus Y_M[1]\ar[rrr]^{\overline{(-)}}&&&(\overline X_M,\overline Y_M)\cong B_{(\overline X_U,\overline Y_U)}^-(\overline X_T,\overline Y_T)}$$

\end{itemize}
\end{theorem}
\begin{proof}
(i) By Proposition \ref{pro-def}, we know that $B_U^-(T)$ is a basic silting object in $\mathcal C$. Then by Theorem \ref{thm:OIY14}, we have that $(\overline X_M,\overline Y_M)$ is a basic $\tau$-tilting pair in $\mod A$.

(ii) By the equivalence of (a) and (e) in Proposition \ref{pro:compare1}, we get that $\Fac \overline X_T$ is a subcategory of  $\prescript{\bot}{}({\tau \overline X_U})\cap \overline Y_U^\bot$. Then by Proposition \ref{pro:finite} (v), we know that the left Bongartz completion 
    $B_{(\overline X_U,\overline Y_U)}^-(\overline X_T,\overline Y_T)$ of $(\overline X_U,\overline Y_U)$ with respect to $(\overline X_T,\overline Y_T)$ exists in $\mod A$.

(iii) By the definition of  $B_{(\overline X_U,\overline Y_U)}^-(\overline X_T,\overline Y_T)$ and Proposition \ref{pro:finite} (v),
we know that  $B_{(\overline X_U,\overline Y_U)}^-(\overline X_T,\overline Y_T)$ is the basic $\tau$-tilting pair corresponding to the functorially finite torsion class $\Fac \overline X_U\ast \Fac \overline X_T$  in $\mod A$. In order to show that $(\overline X_M,\overline Y_M)$ is isomorphic to $B_{(\overline X_U,\overline Y_U)}^-(\overline X_T,\overline Y_T)$, it suffices to show  $$\Fac \overline X_M=\Fac \overline X_U\ast \Fac \overline X_T.$$
Let us first recall the construction of $M=B_U^-(T)$. Take a triangle
$$\xymatrix{T[-1]\ar[r]^f&U^\prime\ar[r]&Z\ar[r]&T}$$ 
with a minimal left $\add U$-approximation $f:T[-1]\rightarrow U^\prime$. Then  $$M=B_U^-(T)=(U\oplus Z)^\flat.$$ Applying the functor $\Hom_{\mathcal D}(S,-)$ to the above triangle, we get the following exact sequence:
$$\Hom_{\mathcal D}(S,U^\prime)\rightarrow \Hom_{\mathcal D}(S,Z)\rightarrow \Hom_{\mathcal D}(S,T)\rightarrow \Hom_{\mathcal D}(S,U^\prime[1]).
$$
By $U^\prime\in \add U\subseteq\mathcal C= S\ast S[1]$, we get $\Hom_{\mathcal D}(S,U^\prime[1])=0$. Notice that for any object $C\in\mathcal C$, we actually have $\Hom_{\mathcal D}(S,C)=\overline C$. Since $U^\prime, Z$ and $T$ are objects in $\mathcal C$, we can rewrite the above exact sequence as follows:
$$\xymatrix{\overline {U^\prime}\ar[r]^{g}&\overline Z\ar[r]^{p}&\overline T\ar[r]&0}.$$
Consider the following short exact sequence:
\begin{eqnarray}\label{eqn:exact}
\xymatrix{0\ar[r]&\overline N\ar[r]^{i}&\overline Z\ar[r]^{p}&\overline T\ar[r]&0},\nonumber
\end{eqnarray}
where $\overline N$ is the image of $g$ and thus $\overline N\in\Fac \overline{U^\prime}\subseteq \Fac\overline{U}$.

Before giving the rest of the proof, let us point out some obvious facts:
$$\overline U=\overline X_U,\;\; \overline T=\overline X_T,\;\; \overline M=\overline X_M.$$

Recall that we want to show $\Fac \overline X_M=\Fac \overline X_U\ast \Fac \overline X_T$. We first show $\Fac \overline X_U\ast \Fac \overline X_T\subseteq \Fac \overline X_M$.  It suffices to show that $\overline X_U$ and $\overline X_T$ are in $\Fac \overline X_M$, because  $\Fac \overline X_M$ is closed under quotient modules and extensions. Because $U$ is in $\add M=\add (U\oplus Z)$, we have
$$\overline X_U=\overline U\in \Fac \overline M=\Fac \overline X_M.$$
Since the morphism $p:\overline Z\rightarrow \overline T$ is surjective and by $Z\in\add M=\add (U\oplus Z)$, we have
$$\overline X_T=\overline T\in\Fac \overline Z\subseteq\Fac \overline M=\Fac \overline X_M.$$
Hence, the inclusion $\Fac \overline X_U\ast \Fac \overline X_T\subseteq \Fac \overline X_M$ holds.

Now let us prove the converse inclusion 
$\Fac \overline X_M\subseteq \Fac \overline X_U\ast \Fac \overline X_T$. It suffices to show that $\overline X_M=\overline M$ is in $\Fac \overline X_U\ast \Fac \overline X_T$, because $\Fac \overline X_U\ast \Fac \overline X_T$ is closed under quotient modules. Clearly, we have $$\overline U=\overline X_U\in \Fac \overline X_U\ast \Fac \overline X_T.$$
Thanks to the short exact sequence $\xymatrix{0\ar[r]&\overline N\ar[r]^{i}&\overline Z\ar[r]^{p}&\overline T\ar[r]&0}$, we know that $\overline Z$ is an extension of $\overline N$ and $\overline T$. Since  $\overline N\in\Fac \overline U=\Fac \overline X_U$ and $\overline T=\overline X_T\in\Fac \overline X_T$, we get
$$\overline Z\in \Fac \overline X_U\ast \Fac \overline X_T.$$
Then by $M\in \add(U\oplus Z)$, we get
 $\overline X_M=\overline M\in\Fac \overline X_U\ast \Fac \overline X_T$.
 So $\Fac \overline X_M\subseteq \Fac \overline X_U\ast \Fac \overline X_T$ and thus $\Fac \overline X_M=\Fac \overline X_U\ast \Fac \overline X_T$.
 
 Hence, the two basic $\tau$-tilting pairs
$(\overline X_M,\overline Y_M)$ and $B_{(\overline X_U,\overline Y_U)}^-(\overline X_T,\overline Y_T)$ correspond to the same functorially finite torsion class in $\mod A$ and thus they are the same up to isomorphisms.
\end{proof}

\section{Remarks on the definition of relative right Bongartz completions}\label{sec:5}

\subsection{Relative right Bongartz completions in $\tau$-tilting theory}
Let $\proj A$ be the full subcatgory of $\mod A$ consisting of projective modules. Denote by $A^{\rm op}$ the opposite algebra of $A$. 

Let  $(-)^\ast:=\Hom_A(-,A):\proj A\rightarrow\proj A^{\rm op}$ be the $A$-dual. For $X\in\mod A$ with a minimal projective presentation
$$\xymatrix{P_1\ar[r]^{d_1} & P_0\ar[r]^{d_0}&X\ar[r]&0}$$
the {\em transpose} $\Tr X\in\mod A^{\rm op}$  of $X$ is defined by  the following exact sequence:
$$\xymatrix{P_0^\ast\ar[r]^{d_1^\ast} & P_1^\ast\ar[r]&\Tr X\ar[r]&0}.$$
It is easy to see that if $X$ is a projective module, then $\Tr X=0$.

The following decomposition will be used frequently without further explanation. For any module $X\in\mod A$, we decompose it as $$X=X_{\rm pr}\oplus X_{\rm np},$$ where $X_{\rm pr}$ is a maximal
projective direct summand of $X$. Such a decomposition is unique up to isomorphisms. For a $\tau$-rigid pair $(U,Q)$ in $\mod A$, let  $$(U,Q)^\dag:=( Q^\ast \oplus \Tr U_{\rm np},U_{\rm pr}^\ast).$$
Clearly, for a $\tau$-rigid pair $(U\oplus U',Q\oplus Q')$, we have $(U\oplus U',Q\oplus Q')^\dag=(U,Q)^\dag\oplus (U',Q')^\dag$.
\begin{theorem}
[\cite{adachi_iyama_reiten_2014}*{Theorem 2.14, Proposition 2.27}]
\label{thm:air214}
\begin{itemize}
    \item [(i)] The map $(-)^\dag:(U,Q)\mapsto (U,Q)^\dag$ induces a bijection from $\tau$-rigid pairs in $\mod A$ to those in $\mod A^{\rm op}$ and one has $$(U,Q)^{\dag\dag}=(U,Q)$$ for any $\tau$-rigid pair $(U,Q)$ in $\mod A$.
    \item[(ii)] The map $(-)^\dag$ restricts to an order-reversing bijection from basic $\tau$-tilting pairs in $\mod A$ to those in $\mod A^{\rm op}$.
\end{itemize}
\end{theorem}

\begin{definition} [Relative right Bongartz completion] 
Let $(U,Q)$ be a basic $\tau$-rigid pair and  $(M,P)$ a basic $\tau$-tilting pair in $\mod A$. If the left Bongartz completion $B_{(U,Q)^\dag}^-(M,P)^\dag$ of $(U,Q)^\dag$ with respect to $(M,P)^\dag$ exists in $\mod A^{\rm op}$, we call the $\tau$-tilting pair $(B_{(U,Q)^\dag}^-(M,P)^\dag)^\dag$  the {\em right Bongart completion} of $(U,Q)$ with respect to $(M,P)$ in $\mod A$ and denote it by $B^+_{(U,Q)}(M,P):=(B_{(U,Q)^\dag}^-(M,P)^\dag)^\dag$.
\end{definition}

If $(M,P)=(A,0)$, then $(M,P)^\dag=(0,A^\ast)$ and thus $B_{(U,Q)^\dag}^-(M,P)^\dag$ is the absolute left Bongartz completion of $(U,Q)^\dag$ in $\mod A^{\rm op}$, which is the unique minimal $\tau$-tilting pair in $\mod A^{\rm op}$ containing $(U,Q)^\dag$ as a direct summand. By Theorem \ref{thm:air214}, we know that $(B_{(U,Q)^\dag}^-(M,P)^\dag)^\dag$ is the unique maximal $\tau$-tilting pair in $\mod A$ containing $(U,Q)$ as a direct summand. In this case, the right Bongartz completion  $B^+_{(U,Q)}(M,P)$ of $(U,Q)$ with respect to $(M,P)=(A,0)$ coincides with the absolute right Bongartz completion of $(U,Q)$.

\subsection{Relative right Bongartz completions in silting theory}
 We fix a $K$-linear, Krull-Schmidt, Hom-finite triangulated category $\mathcal D$ with a basic silting object $S$. Let  $A=\End_{\mathcal D}(S)$ be the endomorphism algebra of $S$ in $\mathcal D$ and denote by $\mathcal C=S\ast S[1]$.

Let us first list some easy facts in the following lemma.

\begin{lemma} 
\begin{itemize}
    \item [(i)] The opposite category $\mathcal D^{\rm op}$ of $\mathcal D$ is still a triangulated category with shift functor $[1]^{\rm op}=[-1]$. Moreover, $$\xymatrix{X\ar[r]&Y\ar[r]&Z\ar[r]&X[1]}$$ is a triangle in $\mathcal D$ if and only if
$$\xymatrix{Z\ar[r]&Y\ar[r]&X\ar[r]&Z[1]^{\rm op}}$$ is a triangle in $\mathcal D^{\rm op}$.
    \item[(ii)] An object $U$ is presilting (respectively, silting) in $\mathcal D$ if and only if it is presilting (respectively, silting) in $\mathcal D^{\rm op}$.

    \item[(iii)] The map $T\mapsto T$ induces an order-reversing bijection from basic silting objects in $\mathcal D$ to those in $\mathcal D^{\rm op}$.
    
    \item[(iv)]  $A^{\rm op}=\End_{\mathcal D^{\rm op}}(S)\cong \End_{\mathcal D^{\rm op}}(S[-1]^{\rm op})$ and $\mathcal C^{\rm op}=(S\ast S[1])^{\rm op}=S[1]\ast_{\rm op} S=S[-1]^{\rm op}\ast_{\rm op} S$, where the symbol $``\ast_{\rm op}"$ means that we consider the extensions in the opposite category $\mathcal D^{\rm op}$.
    \end{itemize}
\end{lemma}

{\em Notations:}
\begin{itemize}
    \item  $\tau\text{-tilt }A$: the set of  isomorphism classes of basic $\tau$-tilting pairs in $\mod A$;
    \item $\text{2-silt }\mathcal C$: the  set of isomorphism classes of basic silting objects in $\mathcal C$;

    \item Denote by $\kappa:\text{2-silt }\mathcal C\rightarrow \tau\text{-tilt }A$ the order-preserving bijection given in Theorem \ref{thm:OIY14};
    \item Denote by  $\kappa^{\rm op}:\text{2-silt }\mathcal C^{\rm op}\rightarrow \tau\text{-tilt }A^{\rm op}$ the order-preserving bijection obtained by applying Theorem \ref{thm:OIY14} to the opposite categories $\mathcal D^{\rm op}$ and $\mathcal C^{\rm op}$.
\end{itemize}

\begin{proposition}\label{pro:compatible}
   The following diagram is commutative.
   $$\xymatrix{\text{2-silt }\mathcal C\ar[d]_{\kappa}\ar[r]^{\rm id}&\text{2-silt }\mathcal C^{\rm op}\ar[d]^{\kappa^{\rm op}}\\
   \tau\text{-tilt }A\ar[r]^{(-)^\dag}&\tau\text{-tilt }A^{\rm op}
   }$$
\end{proposition}

\begin{proof}
Let $T$ be a basic silting object in $\mathcal C$. We decompose $T$ as $$T=T_{-1}\oplus T_0\oplus T_1,$$ where $T_0$ is a maximal direct summand of $T$ belonging to $\add(S)$ and $T_1$ is a maximal direct summand of $T$ belonging to $\add(S[1])$. Then 
\begin{eqnarray}
    \kappa(T)&=&(M_{\rm pr}\oplus M_{\rm np},P)\in \tau\text{-tilt }A,\nonumber\\
     \kappa^{\rm op}(T)&=&(N_{\rm pr}\oplus N_{\rm np},Q)\in \tau\text{-tilt }A^{\rm op},\nonumber
\end{eqnarray}
where
\begin{eqnarray}
   M_{\rm pr}&=&\Hom_{\mathcal D}(S,T_{0}),\;\;M_{\rm np}=\Hom_{\mathcal D}(S,T_{-1}),\;\; P=\Hom_{\mathcal D}(S,T_{1}[-1]),\nonumber \\
    N_{\rm pr}&=&\Hom_{\mathcal D^{\rm op}}(S[-1]^{\rm op},T_{1})=\Hom_{\mathcal D}(T_1,S[1]),\nonumber\\
    N_{\rm np}&=&\Hom_{\mathcal D^{\rm op}}(S[-1]^{\rm op},T_{-1})=\Hom_{\mathcal D}(T_{-1},S[1]),\nonumber\\
    Q&=&\Hom_{\mathcal D^{\rm op}}(S[-1]^{\rm op},T_{0}[-1]^{\rm op})=\Hom_{\mathcal D}(T_{0}[1],S[1]).\nonumber
\end{eqnarray}
By definition, we know that $(M_{\rm pr}\oplus M_{\rm np},P)^\dag=(P^\ast\oplus \Tr M_{\rm np},M_{\rm pr}^\ast)$. In order to show that the given diagram is commutative, it suffices to prove
$(P^\ast\oplus\Tr M_{\rm np}, M_{\rm pr}^\ast)\cong (N_{\rm pr}\oplus N_{\rm np},Q)$.

Since $T_0,T_1[-1]\in \add(S)$, we have
\begin{eqnarray}
  P^\ast&=&\Hom_A(P,A)=\Hom_A(\Hom_{\mathcal D}(S,T_1[-1]),\Hom_{\mathcal D}(S,S))\cong\Hom_{\mathcal D}(T_1[-1],S)\cong N_{\rm pr},\nonumber \\
  M_{\rm pr}^\ast&=&\Hom_A(M_{\rm  
  pr},A)=\Hom_A(\Hom_{\mathcal D}(S,T_0),\Hom_{\mathcal D}(S,S))\cong\Hom_{\mathcal D}(T_0,S)\cong Q.\nonumber
\end{eqnarray}

We still need to show $\Tr M_{\rm np}\cong N_{\rm np}$.
Let $f:S'\rightarrow T_{-1}$ be a right minimal $\add(S)$-approximation of $T_{-1}$ and extend it to a triangle:
\begin{eqnarray}\label{eqn:min-pro}
\xymatrix{S^{\prime\prime}\ar[r]&S'\ar[r]^f&T_{-1}\ar[r]&S^{\prime\prime}[1].}
\end{eqnarray}
Since $T_{-1}$ belongs to $\mathcal C=S\ast S[1]$ and by \cite[Proposition 3.15]{cao_2021}, we have $S^{\prime\prime}\in \add(S)$. Applying the functor $\Hom_{\mathcal D}(S,-)$ to the triangle \eqref{eqn:min-pro}, we get
\begin{eqnarray}\label{eqn:exact-1}
    \xymatrix{\Hom_{\mathcal D}(S,S^{\prime\prime})\ar[r]&\Hom_{\mathcal D}(S,S')\ar[r]&\Hom_{\mathcal D}(S,T_{-1})\ar[r]&0},\nonumber
\end{eqnarray}
which is a minimal projective presentation of $M_{\rm np}=\Hom_{\mathcal D}(S,T_{-1})$ in $\mod A$.

Applying the functor $\Hom_{\mathcal D}(-,S)$ to the triangle \eqref{eqn:min-pro}, we get
\begin{eqnarray}\label{eqn:exact-2}
    \xymatrix{\Hom_{\mathcal D}(S',S)\ar[r]&\Hom_{\mathcal D}(S^{\prime\prime},S)\ar[r]&\Hom_{\mathcal D}(T_{-1}[-1],S)\ar[r]&0},\nonumber
\end{eqnarray}
which is an exact sequence in $\mod A^{\rm op}$.
By definition of $\Tr M_{\rm np}$, we obtain $$\Tr M_{\rm np}=\Hom_{\mathcal D}(T_{-1}[-1],S)\cong N_{\rm np}.$$
Hence, $(M_{\rm pr}\oplus M_{\rm np},P)^\dag= (P^\ast\oplus \Tr M_{\rm np},M_{\rm pr}^\ast)\cong(N_{\rm pr}\oplus N_{\rm np},Q)$ and thus the given diagram is commutative.
\end{proof}

\begin{definition}[Relative right Bongartz completion]
\label{def:right-silting}
Let $S$ be a basic silting object in $\mathcal D$ and denote by $\mathcal C=S\ast S[1]$. Let  $U$ be a basic presilting object in $\mathcal C$ and $T$ a basic silting object in $\mathcal C$.  If the left Bongartz completion $B_U^{\rm op,-}(T)$ of $U$ with respect to $T$ exists in $\mathcal C^{\rm op}$, we call the basic silting object $B_U^{\rm op,-}(T)$ the {\em right Bongartz completion} of $U$ with respect to $T$ in $\mathcal C$ and denote it by $B_U^+(T):=B_U^{\rm op,-}(T)$.
\end{definition}

Keep $U$ and $T$ as above. We take a triangle

$$\xymatrix{T\ar[r]&Y_U\ar[r]&U^\prime\ar[r]^g&T[1]}$$ 
with a minimal right $\add U$-approximation $g:U^\prime\rightarrow T[1]$. It is easy to see that if the right Bongartz completion $B_U^+(T)$ of $U$ with respect to $T$ exists, then it is given by  $B_U^+(T)=(U\oplus Y_U)^\flat$, where  $(U\oplus Y_U)^\flat$  is the basic object corresponding to $U\oplus Y_U$.

\begin{remark}
    Thanks to Proposition \ref{pro:compatible}, we know that the relative right Bongartz completions in silting theory and $\tau$-tilting theory are compatible.
\end{remark}

\bibliographystyle{alpha}
\bibliography{myref}

\end{document}